\pdfoutput=1
\documentclass{amsart}
\usepackage{calc}
\usepackage[nomessages]{fp}
\usepackage{tikz,amsmath,amssymb}
\usepackage{pgfplots}
\usepackage{quiver}
\usepackage{cancel}
\usetikzlibrary{matrix,arrows,backgrounds,shapes.misc,shapes.geometric,patterns,calc,positioning,intersections}
\usetikzlibrary{decorations.pathmorphing, angles, quotes}
\usepgflibrary{patterns}
\pgfplotsset{compat=1.15}
\usepackage{mathrsfs}
\usepackage{rotating}
\usepackage{tkz-euclide} 
\usetikzlibrary{decorations.markings}
\usetikzlibrary{shapes,fit}
\usepackage[english]{babel}
\usepackage[hidelinks]{hyperref}
\usepackage{amsmath,amsfonts,amssymb}
\usepackage{amsthm}
\usepackage[all]{xy}
\usepackage{ stmaryrd }
\usepackage{ mathrsfs }
\theoremstyle{plain} 

\usepackage{epsfig}
\usepackage[T1]{fontenc}
\usepackage{lmodern}
\usepackage{eurosym}
\usepackage{graphicx}
\usepackage{enumitem}
\usepackage{amsbsy}
\usepackage{rotating}
\usepackage[noabbrev,capitalise]{cleveref}
\usepackage[labelsep=period]{caption}
\usepackage{multirow}
\usepackage{xcolor}
\usepackage{nicematrix}

\usepgfplotslibrary{fillbetween}
\makeatletter
\tikzset{use path/.code=\tikz@addmode{\pgfsyssoftpath@setcurrentpath#1}}
\makeatother
\pgfdeclarelayer{bg}
\pgfsetlayers{bg,main}

\usepackage{stackengine,scalerel}
\stackMath
\def\Resleq{\ThisStyle{\mathrel{%
  \stackinset{r}{.75pt+.15\LMpt}{t}{.1\LMpt}{\rule{.3pt}{1.1\LMex+.2ex}}{\SavedStyle\leqslant}%
}}}

\def\Accordleq{\ThisStyle{\mathrel{%
  \stackinset{r}{.75pt+.15\LMpt}{t}{.1\LMpt}{\rule{.3pt}{1.1\LMex+.2ex}}{\SavedStyle\preccurlyeq}%
}}}

\definecolor{midnightblue}{rgb}{0.1, 0.1, 0.44}
\definecolor{plum}{rgb}{0.56, 0.27, 0.52}
\definecolor{Plum}{rgb}{0.56, 0.27, 0.52}
\definecolor{patriarch}{rgb}{0.5, 0.0, 0.5}
\definecolor{darkgreen}{rgb}{0.0, 0.2, 0.13}
\definecolor{darkcerulean}{rgb}{0.03, 0.27, 0.49}
\definecolor{jade}{rgb}{0.0, 0.66, 0.42}

\newcommand{\Anti}{\operatorname{Anti}}

\newcommand{\ArExt}{\operatorname{ArExt}}
\newcommand{\OvExt}{\operatorname{OvExt}}
\newcommand{\s}{\operatorname{\pmb{s}}}
\renewcommand{\t}{\operatorname{\pmb{t}}}

\newcommand{\wt}{\mathbf{wt}}

\makeatletter
\newcommand{\addbar}[3]{{\vphantom{#3}\mathpalette\add@bar{{#1}{#2}{#3}}}}
\newcommand{\add@bar}[2]{\add@@bar{#1}#2}
\newcommand{\add@@bar}[4]{%
  \begingroup
  \sbox\z@{$\m@th#1#4$}%
  \ooalign{%
    \hidewidth\kern#2\wd\z@\add@@@bar{#1}{#3}\hidewidth\cr
    \box\z@\cr
  }%
  \endgroup
}
\newcommand{\add@@@bar}[2]{%
  \sbox\tw@{$\m@th#1\newmcodes@\if\relax#2\relax-\else\bm{-}\fi$}%
  \raisebox{\dimexpr(\ht\z@-\ht\tw@)/2}{\usebox\tw@}%
}
\makeatother

\newcommand{\bbGamma}{\pmb{\mathpalette\makebbGamma\relax}}
\newcommand{\makebbGamma}[2]{%
  \raisebox{\depth}{\scalebox{1}[-1]{$\mathsurround=0pt#1\mathbb{L}$}}%
}

\newcommand\precdot{\mathrel{\ooalign{$\prec$\cr
  \hidewidth\raise0ex\hbox{$\cdot\mkern0.5mu$}\cr}}}

\hypersetup{
	colorlinks=true,
	linkcolor=darkcerulean,
	citecolor=jade,
	filecolor=magenta,      
	urlcolor=cyan,
	pdfpagemode=FullScreen,
}

\renewcommand{\mod}{\operatorname{mod}}

\newcommand{\rep}{\operatorname{rep}}
\newcommand{\add}{\operatorname{add}}
\newcommand{\Ker}{\operatorname{Ker}}

\newcommand{\Hom}{\operatorname{Hom}}

\newcommand{\Ext}{\operatorname{Ext}}

\newcommand{\ind}{\operatorname{ind}}
\newcommand{\proj}{\operatorname{proj}}

\newcommand{\Res}{\operatorname{Res}}
\newcommand{\ResOrd}{\operatorname{\pmb{Res}}}

\newcommand{\THomleq}{\pmb{\pmb{\rightarrow}}}

\newcommand{\opQ}{\operatorname{\mathbf{Q}}}
\newcommand{\opR}{\operatorname{\mathbf{R}}}

\newcommand{\opResAc}{\operatorname{\pmb{\mathscr{R}}}}
\newcommand{\Prj}{\operatorname{Prj}}
\newcommand{\MM}{\operatorname{M}}
\newcommand{\NP}{\operatorname{N}_{\proj}}

\newcommand{\Surf}{\operatorname{\pmb{\mathcal{S}}}}

\newcommand{\tc}{t_{\operatorname{cell}}}
\renewcommand{\sc}{s_{\operatorname{cell}}}
\renewcommand{\Ker}{\operatorname{Ker}}
\newcommand{\cl}{\operatorname{\mathsf{cl}}}

\newcommand{\col}{\operatorname{\mathsf{col}}}

\newcommand{\Accord}{\operatorname{\pmb{\mathscr{A}}}}

\newcommand{\CoZ}{\operatorname{Co-\mathbf{Z}}}

\newcommand{\X}{\operatorname{\mathbf{X}}}

\newcommand{\Tilt}{\operatorname{Tilt}}
\newcommand{\Extperp}{\pmb{\pmb{\perp}}}

\newcommand{\bleq}{\mathrel{\mathpalette\bleqinn\relax}}
\newcommand{\bleqinn}[2]{%
  \ooalign{%
    \raisebox{.2ex}{$#1\blacktriangleleft$}\cr
    $#1\leqslant$\cr
  }%
}

\author[B.~Dequêne]{Benjamin Dequêne}
\address[B.~Dequêne]{School of Mathematics, University of Leeds}
\email{B.D.Dequene@leeds.ac.uk}

\author[M.~Schoonheere]{Michael Schoonheere}
\address[M.~Schoonheere]{LAMFA, Université de Picardie Jules Verne}
\email{michael.schoonheere@u-picardie.fr}

\title[Resolving subcategories for gentle algebras II]{Resolving subcategories for gentle algebras III: Tilting modules for gentle tree algebras}

\date{\today}

\usepackage{thmtools}
\declaretheorem[numberwithin=section,name=Theorem,
refname={Theorem,Theorems},
Refname={Theorem,Theorems}]{theorem}
\declaretheorem[numberlike=theorem,name=Lemma,
refname={Lemma,Lemmas},
Refname={Lemma,Lemmas}]{lemma}
\declaretheorem[numberlike=theorem,name=Proposition,
refname={Proposition,Propositions},
Refname={Proposition,Propositions}]{prop}
\declaretheorem[numberlike=theorem,name=Corollary,
refname={Corollary,Corollaries},
Refname={Corollary,Corollaries}]{cor}

\declaretheorem[style=definition,numberlike=theorem,name=Definition,
refname={Definition,Definitions},
Refname={Definition,Definitions}]{definition}
\declaretheorem[style=definition,numberlike=theorem,name=Convention,
refname={Convention,Conventions},
Refname={Convention,Conventions}]{conv}
\declaretheorem[style=definition,numberlike=theorem,name=Algorithm,
refname={Algorithm,Algorithms},
Refname={Algorithm,Algorithms}]{algo}

\declaretheorem[style=definition,numberlike=theorem,name=Example,
refname={Example,Examples},
Refname={Example,Examples}]{ex}

\declaretheorem[style=remark,numberlike=theorem,name=Remark,
refname={Remark,Remarks},
Refname={Remark,Remarks}]{remark}

\usepackage{todonotes}

\definecolor{aquamarine}{rgb}{0.5, 1.0, 0.83}

\newcommand{\new}[1]{\textit{\textbf{\color{patriarch}{#1}}}}

\definecolor{dark-green}{RGB}{14,150,2}
\newcommand{\gpoint}{{\color{dark-green}{\circ}}}
\newcommand{\rpoint}{\textcolor{red}{\bullet}}
\definecolor{darkgray}{rgb}{0.66, 0.66, 0.66}
\definecolor{darkpink}{rgb}{0.91, 0.33, 0.5}
\newcommand{\gsquare}{\color{dark-green}{\pmb{\square}}}
\newcommand{\rsquare}{\color{red}{{\blacksquare}}}
\newcommand{\osquare}{\color{orange}{\pmb{\boxtimes}}}
\newcommand{\psquare}{\color{darkpink}{\pmb{\times}}}

\definecolor{mypurple}{rgb}{0.63, 0.36, 0.94}

\usepackage{pict2e,picture}
\makeatletter
\DeclareRobustCommand{\bbDelta}{{\mathpalette\bb@Delta\relax}}
\newcommand{\bb@Delta}[2]{%
  \begingroup
  \sbox\z@{$\m@th#1\Delta$}%
  \dimendef\Dht=6 \dimendef\Dwd=8
  \setlength{\Dwd}{\wd\z@}%
  \setlength{\Dht}{\ht\z@}%
  \begin{picture}(\Dwd,\Dht)
  \put(0,0){$\m@th#1\Delta$}
  \put(.42\Dwd,.7\Dht){\line(10,-26){.25\Dht}}
  \end{picture}%
  \endgroup
}
\makeatother

\begin{document}

\begin{abstract}
This paper is the third part of a series that intends to study the resolving subcategories for gentle algebras over an algebraically closed field $\mathbb{K}$.

As in the previous two papers, we continue to focus on gentle trees $(Q,R)$. Via a modified surface model for gentle algebras with finite global dimension, we developed combinatorial, poset, and quiver representation techniques that allow one to calculate all the resolving subcategories of $\mathbb{K}Q/\langle R \rangle$-mod. Furthermore, they enable one to calculate the resolving subcategory generated by any collection of $\mathbb{K}Q/\langle R \rangle$-modules.

In this paper, based on those techniques, we give a combinatorial realization of the Auslander--Reiten one-to-one correspondence between resolving subcategories and tilting modules in $\mathbb{K}Q/\langle R \rangle$-mod.
\end{abstract}
\maketitle
    
\tableofcontents

	\section{Introduction}
	\label{sec:Intro}
	\pagestyle{plain}

Resolving subcategories were first mentioned by Auslander and Bridger \cite{Auslander1969}, to describe the subcategory of modules of Gorenstein dimension 0. They are defined as additive subcategories that contain the projective modules and are closed under extensions and kernels of epimorphisms. Later works \cite{Auslander1991, Adachi2022} focused on the functorially finite resolving subcategories and showed that they are in one-to-one correspondence with hereditary complete cotorsion pairs. Since this bijection preserves inclusion, the poset of resolving subcategories provides an alternative approach to studying the poset of hereditary cotorsion pairs.

In the case of abelian (resp. extriangulated) categories, under some stronger assumptions,  Auslander and Reiten (resp. Adachi and Tsukamoto) highlighted that resolving subcategories are in bijection with tilting (resp. silting) objects. Those objects naturally appear in the context of exceptional sequences in \cite{Ringel91} and their characteristic tilting, which is closely linked to quasi-hereditary algebras. To have information on tilting objects, one has to compute all resolving subcategories. One way to achieve this is to compute the smallest subcategory that contains the set of objects $\mathcal{X}$. Takahashi addressed this issue in \cite{T09}, where he gave a procedure for computing such a resolving subcategory. This category is called the resolving closure of $\mathcal{X}$. The given procedure works in the broad setting of the module category of an algebra of finite representation type. It thus does not benefit from the combinatorial insights provided by the algebra. 

Gentle algebras are algebras with a strong combinatorial structure. They arise from quivers in relations and first appeared in \cite{Assem1981,Assem1982}. The combinatorial behavior of gentle algebras is widely studied: extensions of indecomposable objects have at most two indecomposable summands \cite{BS80,WW85,GP68,SW83,DS87}, and the hom-spaces and extensions are endowed with a basis arising from string combinatorics. To complete the combinatorics, geometric models consisting of dissected surfaces correspond to the module category over the path algebra, the derived category, and the two-term homotopy category of projective \cite{BCS21,OPS18,APS19,PPP18,PPP182,chang2025}. These models offer more straightforward categorical computations.

In previous work, \cite{DS251,DS252}, we gave an explicit computation of all resolving subcategories of $\mod{A}$ for $A$ a gentle algebra arising from oriented trees. The first step was to describe the join-irreducible elements of the poset of resolving subcategories ordered by inclusion. These objects are precisely the subcategories generated by a unique non-projective indecomposable object. We call these subcategories monogeneous resolving subcategories. We compute these subcategories using the geometric model. The poset of resolving subcategories is not semidistributive in general. Thus, we defined a canonical antichain for each resolving subcategory, such that a resolving subcategory corresponds to an ideal generated by this antichain in the poset.

The main result of this paper is the following: using that description, we can recover the tilting object associated with a resolving subcategory in \cite{Auslander1991} combinatorially.

\begin{theorem} \label{thm:Tiltgeneratorintro}
Let $(Q,R)$ be a gentle tree, and $\mathscr{R}$ be a resolving category described by the antichain $\mathfrak{D}$. Then, adding the indecomposable projective objects $\Ext$-orthogonal to $\mathfrak{D}$ and completing the object following an algorithmic procedure, we obtain $\mathfrak{G} \subset \Accord'$ such that both:
\begin{enumerate}[label=$(\roman*)$, itemsep=1mm]
    \item $\mathfrak{T} = \{\MM(\delta) \mid \delta \in \mathfrak{G}\}$ is a tilting collection of $(Q,R)$; and,
    \item $\Res(\mathfrak{T}) = \mathscr{R}$.
\end{enumerate}
Therefore, $\mathfrak{T}$ coincides with $\mathfrak{T}_{\mathscr{R}}$ defined via the Auslander-Reiten correspondence recalled in \cref{thm:Auslander}.
\end{theorem}

\cref{sec:Setting} provides some preliminaries on gentle quivers and their geometric models. This section also recalls what is known about tilting objects in the setup of gentle algebras. In \cref{sec:ResSubcats}, we recall the results and notation from previous work that we will use here.
\cref{sec:Tilt} describes the computation of the tilting object associated with a resolving subcategory. After recalling the bijection of \cite{Auslander1991}, we first construct a canonical rigid object that generates the resolving subcategory in \cref{ss:canonrigid}, and then complete it to a tilting generator in \cref{ss:CombAuslander}. We give detailed computations of tilting generators in \cref{ss:examplePart2}.
	

        \section{The setting}
        \label{sec:Setting}
        \pagestyle{plain}

\subsection{Gentle quivers and geometric model}
\label{ss:gentle}
Fix an algebraically closed field $\mathbb{K}$. A \new{gentle quiver} is a pair $(Q,R)$ such that:
	\begin{enumerate}[label=$\bullet$, itemsep=1mm]
		\item $Q$ is a quiver such that there are at most two incoming arrows and at most two outgoing arrows at each vertex of $Q$;
		\item  $R$ is a set of quadratic relations such that: 
		\begin{enumerate}[label=$\bullet$,itemsep=1mm]
			\item for any arrow $\alpha \in Q_1$:
			\begin{enumerate}[label = $\bullet$, itemsep=1mm]
				\item there is at most one $\beta \in Q_1$ such that $s(\beta) = t(\alpha)$ and $\beta \alpha \in R$;
				\item there is at most one $\gamma \in Q_1$ such that $s(\gamma) = t(\alpha)$ and $\gamma \alpha \notin R$;
				\item there is at most one $\beta' \in Q_1$ such that $t(\beta') = s(\alpha)$ and $\alpha \beta' \in R$;
				\item there is at most one $\gamma' \in Q_1$ such that $t(\gamma') = s(\alpha)$ and $\alpha \gamma' \notin R$;
			\end{enumerate}
		\item for any cycle $c = \alpha_k \cdots \alpha_1$, either there exists $i \in \{1, \ldots, k-1\}$ such that $\alpha_{i+1} \alpha_i \in R$, or $\alpha_1 \alpha_k \in R$.
		\end{enumerate}
	\end{enumerate}
A $\mathbb{K}$-algebra $\Lambda$ is said to be \new{gentle} if there exists a gentle quiver $(Q,R)$ such that $\Lambda \cong \mathbb{K}Q/\langle R \rangle$. The category $\rep(Q,R)$ of finite-dimensional representations of $(Q,R)$ (over $\mathbb{K}$) is equivalent to the category of finitely-generated left $\mathbb{K}Q/\langle R \rangle$-modules. Therefore $\rep(Q,R)$ is abelian, Krull--Schmidt, and has enough projective objects. Denote by $\ind(Q,R)$ the collection of the (isomorphism classes of) indecomposable representations of $(Q,R)$. It is possible to describe $\rep(Q,R)$ using string combinatorics \cite{BR87,CB89} by focusing on representation-finite gentle quivers. Morphisms, and thus syzygies, are also depicted combinatorially, as well as extensions between objects in $\ind(Q,R)$. In particular, \c{C}anak\c{c}i, Pauksztello, and Schroll \cite{CPS21} and Brüstle, Douville, Mousavand, Thomas and Y\i ld\i r\i m \cite{BDMTY19} give an explicit basis of $\Ext^1(M,N)$ in terms of arrow and overlap extensions, for any $M,N \in \ind(Q,R)$. \cite{BCS21,OPS18,PPP18} introduced geometric models for gentle quivers. Let us first state some conventions and definitions.

\begin{conv}
Whenever we talk about \new{surfaces}, we always mean an oriented compact surface $\pmb{\Sigma}$ with a finite number of open discs (possibly zero) removed. Write $\partial \pmb{\Sigma}$ for the boundary of $\pmb{\Sigma}$. Such a surface is determined, up to homeomorphism, by its genus and by the number of connected components of $\partial \pmb{\Sigma}$; we will refer to these connected components as the \new{boundary components} of $\pmb{\Sigma}$.
\end{conv}

\begin{definition} \label{def:marksurf}
A \new{marked surface} is a pair $(\pmb{\Sigma},\mathcal{M})$, where $\pmb{\Sigma}$ is a surface and $\mathcal{M}$ is a finite set of \new{marked points} of $\pmb{\Sigma}$ such that $\mathcal{M}$ admits a bipartition $\{\mathcal{M}_{\gpoint},\mathcal{M}_{\rpoint}\}$ such that, on each boundary component of $\pmb{\Sigma}$: 
\begin{enumerate}[label = $\bullet$, itemsep=0.1em]
	\item there is at least one marked point of each color;
	\item marked points in $\mathcal{M}_{\gpoint}$ and $\mathcal{M}_{\rpoint}$ alternate.
\end{enumerate}
\end{definition}

Let $(\pmb{\Sigma},\mathcal{M})$ be a marked surface.
\begin{enumerate}[label=$\bullet$,itemsep=1mm]
\item A \new{$\gpoint$-arc} is a curve, up to homotopy, on $(\pmb{\Sigma},\mathcal{M})$ joining two points in $\mathcal{M}_{\gpoint}$; it is a continuous map from the closed interval $[0,1]$ to $\pmb{\Sigma}$ with endpoints in $\mathcal{M}_{\gpoint}$, and such that the image of its interior is disjoint from $\mathcal{M}$.

\item A $\gpoint$-arc is said to be \new{simple} if it does not intersect itself (except perhaps at its endpoints). 

\item A \new{$\gpoint$-dissection} of $(\pmb{\Sigma}, \mathcal{M})$ is a collection $\Delta^{\gpoint}$ of pairwise non-intersecting simple $\gpoint$-arcs which cut the surface into polygons (that is to say, into simply-connected regions), called the \new{cells} of the dissection. The triplet $(\pmb{\Sigma}, \mathcal{M}, \Delta^{\gpoint})$  is called a  \new{$\gpoint$-dissected marked surface}.

\item A cell is said to be \new{large} if the polygon has more edges than a triangle.

\item A \new{$\gpoint$-dissection} $\Delta^{\gpoint}$ of $(\pmb{\Sigma}, \mathcal{M})$ is said to be \new{admissible} if each cell of $\Delta^{\gpoint}$ contains at least one marked point in $\mathcal{M}_{\rpoint}$. In the case where each cell of $\Delta^{\gpoint}$ contains exactly one marked point in $\mathcal{M}_{\rpoint}$, we say that $\Delta^{\gpoint}$ is \new{dualizable}.

\item We define \new{$\rpoint$-arcs} and \new{$\rpoint$-dissection} similarly.

\item Given a $\gpoint$-dissection or a $\rpoint$-dissection $\Delta$ of a marked surface $(\pmb{\Sigma}, \mathcal{M})$, we denote by $\pmb{\Gamma}(\Delta)$ the set of cells of $(\pmb{\Sigma}, \mathcal{M},\Delta)$. 
\end{enumerate}

\begin{remark} \label{rem:previous} Contrary to what we have done previously, we distinguish between $\gpoint$-dissections that are dualizable, and those that are not. This will be useful later on.
\end{remark}

\cite{BCS21,OPS18,PPP18} established a one-to-one correspondence between gentle quivers and $\gpoint$-dissected marked surfaces, which is uniquely determined up to oriented homeomorphisms and homotopies of $\gpoint$-arcs. We write $\Surf(Q,R) = (\pmb{\Sigma}, \mathcal{M}, \Delta^{\gpoint})$ for the $\gpoint$-dissected marked surface associated to the gentle quiver $(Q,R)$. Moreover, we can translate information from $\rep(Q,R)$ to the geometric model.

\begin{theorem}[\cite{BCS21,PPP182}] \label{thm:GeomandRep} Consider  a dualizable $\gpoint$-dissected marked surface $(\pmb{\Sigma}, \mathcal{M}, \Delta^{\gpoint})$ such that $\mathcal{M}_{\gpoint} \subset \partial \pmb{\Sigma}$, and let $(\opQ(\Delta^{\gpoint}),\opR(\Delta^{\gpoint}))$ be its associated gentle quiver. Then we have a one-to-one correspondence from indecomposable representations of  $(\opQ(\Delta^{\gpoint}),\opR(\Delta^{\gpoint}))$, up to isomorphism, to a subfamily of $\rpoint$-curves, called \emph{accordions} on $(\pmb{\Sigma}, \mathcal{M}, \Delta^{\gpoint})$, up to homotopy.
\end{theorem}

Write $\Accord = \mathscr{A}(\pmb{\Sigma}, \mathcal{M}, \Delta^{\gpoint})$ for the set of accordions on $(\pmb{\Sigma}, \mathcal{M}, \Delta^{\gpoint})$. Recall that:
\begin{enumerate}[label=$\bullet$, itemsep=1mm]
    \item for any $M \in \ind(Q,R)$, we write $\gamma_{(M)}$ for its associated accordion up to homotopy; and,
    \item for any $\delta \in \Accord$, we denote by $\MM(\delta)$ its associated indecomposable representations up to isomorphism.
\end{enumerate} We can also describe morphisms and extensions between indecomposable representations as geometric behavior between accordions. We recall the precise statement.

\begin{prop}[\cite{BR87,BCS21}] \label{prop:HomCrossing}
Let $(\pmb{\Sigma},\mathcal{M},\Delta^{\gpoint})$ be a $\gpoint$-dissected marked surface with $\mathcal{M}_{\gpoint} \subset \partial{\pmb{\Sigma}}$. For any pair $(\delta, \eta) \in \Accord$, the basis elements of $\Hom(\MM(\delta), \MM(\eta))$ are given by crossings as depicted in \cref{fig:HomCrossing}.
\begin{figure}[ht!]
    \centering
    \begin{tikzpicture}[mydot/.style={
				circle,
				thick,
				fill=white,
				draw,
				outer sep=0.5pt,
				inner sep=1pt
			}, fl/.style={->,>=latex}]
			\tkzDefPoint(0,0){O}\tkzDefPoint(0,1.5){1} 
			\tkzDefPoint(-1.5,1.5){2} 
			\tkzDefPoint(-1.5,0){3} 
			
			\tkzDefPoint(3,0){4}\tkzDefPoint(3,1.5){5} 
			\tkzDefPoint(4.5,1.5){6} 
			\tkzDefPoint(4.5,0){7} 
			
			\filldraw [fill=gray,opacity=0.1] (O) to (1) to [bend right=10] (5) to (4) to [bend right=10] cycle;
			
			\draw[line width=0.5mm,dark-green](1) edge (O);
			\draw[line width=0.5mm,dark-green, bend left=30](1) edge (2);
			\draw[line width=0.5mm,dark-green,bend right=30](O) edge (3);
			
			\draw[line width=0.5mm,dark-green](4) edge (5);
			\draw[line width=0.5mm,dark-green, bend left=30](6) edge (5);
			\draw[line width=0.5mm,dark-green,bend right=30](7) edge (4);
			
			\draw[line width=0.7mm,orange,densely dotted](-0.75,1.5) to [bend right=40] (0,0.85) to [bend left=0] node[above]{$\delta$}  (3,0.85) to [bend left=40] (3.75,0);
			\draw[line width=0.7mm,blue, bend left=30, loosely dotted](-0.75,0) to [bend left=30]  (0,0.65) to [bend left=0] node[below]{$\eta$} (3,0.65) to [bend right=30] (3.75,1.5);
			
			\tkzDrawPoints[size=4,color=dark-green,mydot](O,1,2,3,4,5,6,7);
		\end{tikzpicture}
    \caption{\label{fig:HomCrossing} Illustration of a crossing corresponding to a basis element of $\Hom(\MM(\delta), \MM(\eta))$. The shaded part is where all the segments of $\delta$ and all the ones of $\eta$, given by the cutting of $\pmb{\Sigma}$ with $\pmb{\Gamma}( Gammaa^{\gpoint})$, are homotopic.}
\end{figure}
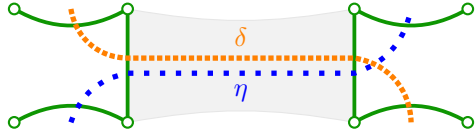
\end{prop}

\begin{conv}
Whenever we say that two arcs cross, they cross in their relative interior.
\end{conv}
	
\begin{prop}[\cite{DS251}]\label{prop:geo_mor}
Let $(\pmb{\Sigma},\mathcal{M},\Delta^{\gpoint})$ be a dualizable $\gpoint$-dissected marked surface with $\mathcal{M}_{\gpoint} \subset \partial{\pmb{\Sigma}}$. Consider $p \in \mathbb{N}^*$ and $\delta_1,\ldots, \delta_p \in \Accord$. Let $\eta \in \Accord$ be such that there exists a minimal epimorphism \[ \begin{tikzcd}
	{\displaystyle f: \bigoplus_{i=1}^p \MM(\delta_i)} & \MM(\eta)
	\arrow[two heads, from=1-1, to=1-2]
\end{tikzcd}\]
where each $f_{|\MM(\delta_i)}$ is given by exactly one crossing between $\delta_i$ and $\eta$. Then \[\Ker(f) = \bigoplus_{i=0}^p \MM(\kappa_i)\] where $\kappa_0,\ldots,\kappa_p \in \Accord$  are constructed from $\delta_1, \ldots, \delta_p$, and $\eta$ in \cref{fig:kerepi} Note that $\kappa_0$ is defined in a dual way to $\kappa_p$ and thus the drawing of $\kappa_p$ is dual to the one of $\kappa_0 $.
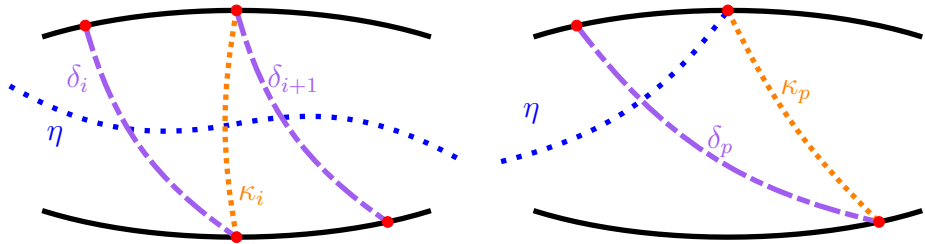
\begin{figure}[!ht]
\centering 
     \begin{tikzpicture}[mydot/.style={
					circle,
					thick,
					fill=white,
					draw,
					outer sep=0.5pt,
					inner sep=1pt
				}, scale = 1]
		\tikzset{
		osq/.style={
        rectangle,
        thick,
        fill=white,
        append after command={
            node [
                fit=(\tikzlastnode),
                orange,
                line width=0.3mm,
                inner sep=-\pgflinewidth,
                cross out,
                draw
            ] {}}}}
		\draw [line width=0.7mm,domain=50:130] plot ({4*cos(\x)}, {1.5*sin(\x)});
        \draw [line width=0.7mm,domain=230:310] plot ({4*cos(\x)}, {1.5*sin(\x)});
		\foreach \X in {0,1}
		{
		\tkzDefPoint(4*cos(pi/6*\X +pi/2),1.5*sin(pi/6*\X + pi/2)){\X};
		};
		\foreach \X in {2,3}
		{
		\tkzDefPoint(4*cos(pi/6*(\X-2) +3*pi/2),1.5*sin(pi/6*(\X-2) + 3*pi/2)){\X};
		};
		
		\draw[line width=0.7mm ,bend right=20,blue, loosely dotted](-3,0.5) edge (0,0);
		
		\draw[line width=0.7mm ,bend left=20,blue, loosely dotted](0,0) edge (3,-0.5);
		
		\draw[line width=0.7mm ,bend right=10,orange,dotted](0) edge (2);

		\draw [line width=0.7mm, mypurple,dash pattern={on 10pt off 2pt on 5pt off 2pt}, bend right=20] (1) edge (2);
		\draw [line width=0.7mm, mypurple,dash pattern={on 10pt off 2pt on 5pt off 2pt}, bend right=20] (0) edge (3);
		
		\foreach \X in {0,1,2,3}
		{
		\tkzDrawPoints[fill =red,size=4,color=red](\X);
		};
		
		

		\tkzDefPoint(-2.4,0.1){gammaM};
		\tkzLabelPoint[blue](gammaM){\Large $\eta$}
		\tkzDefPoint(0.2,-0.7){gammaP};
		\tkzLabelPoint[orange](gammaP){\Large $\kappa_i$}
		\tkzDefPoint(-2.1,0.9){deltav};
		\tkzLabelPoint[mypurple](deltav){\Large $\delta_i$}
		\tkzDefPoint(0.75,0.9){deltaw};
		\tkzLabelPoint[mypurple](deltaw){\Large $\delta_{i+1}$}
		
		\begin{scope}[xshift = 6.5cm]
		\draw [line width=0.7mm,domain=50:130] plot ({4*cos(\x)}, {1.5*sin(\x)});
        \draw [line width=0.7mm,domain=230:310] plot ({4*cos(\x)}, {1.5*sin(\x)});
		\foreach \X in {0,1}
		{
		\tkzDefPoint(4*cos(pi/6*\X +pi/2),1.5*sin(pi/6*\X + pi/2)){\X};
		};
		\foreach \X in {2,3}
		{
		\tkzDefPoint(4*cos(pi/6*(\X-2) +3*pi/2),1.5*sin(pi/6*(\X-2) + 3*pi/2)){\X};
		};
		
		\draw[line width=0.7mm ,bend right=20,blue, loosely dotted](-3,-0.5) edge (0);
		
		\draw [line width=0.7mm, mypurple,dash pattern={on 10pt off 2pt on 5pt off 2pt}, bend right=20] (1) edge (3);
		
		\draw[line width=0.7mm ,bend right=10,orange,dotted](0) edge (3);
		
		\foreach \X in {0,1,3}
		{
		\tkzDrawPoints[fill =red,size=4,color=red](\X);
		};
		
		

		\tkzDefPoint(-2.6,0.4){gammaM};
		\tkzLabelPoint[blue](gammaM){\Large $\eta$}
		\tkzDefPoint(0.9,0.7){gammaP};
		\tkzLabelPoint[orange](gammaP){\Large $\kappa_p$}
		\tkzDefPoint(-0.1,0.1){deltav};
		\tkzLabelPoint[mypurple](deltav){\Large $\delta_p$}
		\end{scope}
    \end{tikzpicture}
\caption{\label{fig:kerepi} The two types of accordions $\kappa_i$ representing the indecomposable summands of $\Ker(f)$.}
\end{figure}
\end{prop}

\begin{prop}[\cite{DS251}] \label{prop:geom_ext}
Let $(\pmb{\Sigma}, \mathcal{M},\Delta^{\gpoint})$ be a dualizable $\gpoint$-dissected marked surface with $\mathcal{M}_{\gpoint} \subset \partial \pmb{\Sigma}$. Let $\delta, \eta \in \Accord$. We distinguish betweenn two types of extensions from $\MM(\delta)$ to $\MM(\eta)$ which are:

\begin{enumerate}[label = $\bullet$, itemsep =0.5em]
    \item \new{Overlap extensions} : whenever we have a nonsplit short exact sequence \[\begin{tikzcd}
	 \MM(\eta) & {E_1 \oplus E_2} & \MM(\delta),
	\arrow[tail, from=1-1, to=1-2]
	\arrow[two heads, from=1-2, to=1-3]
\end{tikzcd}\] where $E_1,E_2 \in \ind_\mathbb{K}(Q,R)$, then $\delta$ and $\eta$ are intersecting each other, and $\gamma_{(E_1)}$ and $\gamma_{(E_2)}$ can be obtained from $\delta$ and $\eta$ as pictured in \cref{fig:overlapextaccord}. By abuse of notations with strings of $(Q,R)$, we write $\OvExt(\delta,\eta)$ for the union of all sets $\{\gamma_{(E_1)}, \gamma_{(E_2)}\}$ over all the equivalence classes of short exact sequences of the above shape.
   \begin{figure}[!ht]
\centering 
    \begin{tikzpicture}[mydot/.style={
					circle,
					thick,
					fill=white,
					draw,
					outer sep=0.5pt,
					inner sep=1pt
				}, scale = 1]
		\tikzset{
		osq/.style={
        rectangle,
        thick,
        fill=white,
        append after command={
            node [
                fit=(\tikzlastnode),
                orange,
                line width=0.3mm,
                inner sep=-\pgflinewidth,
                cross out,
                draw
            ] {}}}}
		\draw [line width=0.7mm,domain=50:130] plot ({4*cos(\x)}, {1.5*sin(\x)});
        \draw [line width=0.7mm,domain=230:310] plot ({4*cos(\x)}, {1.5*sin(\x)});
		\foreach \X in {0}
        {
		\tkzDefPoint(4*cos(pi/6*\X +pi/3),1.5*sin(pi/6*\X + pi/3)){\X};
		};
        \foreach \X in {1}
		{
		\tkzDefPoint(4*cos(pi/6*\X +pi/2),1.5*sin(pi/6*\X + pi/2)){\X};
		};
        \foreach \X in {2}
        {
		\tkzDefPoint(4*cos(pi/6*(\X-2.1) +3*pi/2.2),1.5*sin(pi/6*(\X-2.1) + 3*pi/2.2)){\X};
		}
		\foreach \X in {3}
		{
		\tkzDefPoint(4*cos(pi/6*(\X-2.1) +3*pi/2),1.5*sin(pi/6*(\X-2.1) + 3*pi/2)){\X};
		};
        
        \tkzDefPoint(4*cos(pi/6*0 +pi/3.5),1.5*sin(pi/6*0 + pi/3.5)){e0};
        \tkzDefPoint(4*cos(pi/6*0 +pi/2.6),1.5*sin(pi/6*0 + pi/2.6)){e0'};
         \tkzDefPoint(4*cos(pi/6*1 +pi/1.85),1.5*sin(pi/6*1 + pi/1.85)){e1};
        \tkzDefPoint(4*cos(pi/6*1 +pi/2.2),1.5*sin(pi/6*1 + pi/2.2)){e1'};
        \tkzDefPoint(4*cos(pi/6*(2-2.1) +3*pi/2.28),1.5*sin(pi/6*(2-2.1) + 3*pi/2.28)){e2};
        \tkzDefPoint(4*cos(pi/6*(2-2.1) +3*pi/2.15),1.5*sin(pi/6*(2-2.1) + 3*pi/2.15)){e2'};
        \tkzDefPoint(4*cos(pi/6*(3-2.1) +3*pi/1.94),1.5*sin(pi/6*(3-2.1) + 3*pi/1.94)){e3};
        \tkzDefPoint(4*cos(pi/6*(3-2.1) +3*pi/2.05),1.5*sin(pi/6*(3-2.1) + 3*pi/2.05)){e3'};

        \filldraw [fill=gray,opacity=0.1] (e1') to (e2') to [bend left=14] (e3') to (e0') to [bend left=15] cycle;

        \draw[line width=0.5mm,dark-green](e0') edge (e3');
        \draw[line width=0.5mm,dark-green](e1') edge (e2');

        \draw[line width=0.5mm,dark-green, bend left=30](e0) edge (e0');
        \draw[line width=0.5mm,dark-green, bend left=30](e1') edge (e1);
        \draw[line width=0.5mm,dark-green, bend left=30](e2) edge (e2');
        \draw[line width=0.5mm,dark-green, bend left=30](e3') edge (e3);

		\draw[line width=0.7mm,bend right=5,orange, dotted](0) edge (2);
		
		\draw[line width=0.7mm,bend right=10,blue, loosely dotted](1) edge (3);
		
		\draw [line width=0.7mm, mypurple,dash pattern={on 10pt off 2pt on 5pt off 2pt}, bend right=20] (1) edge (0);
		\draw [line width=0.7mm, mypurple,dash pattern={on 10pt off 2pt on 5pt off 2pt}, bend right=20] (3) edge (2);
		\foreach \X in {0,...,3}
		{
		\tkzDrawPoints[circle,fill =red,size=4,color=red](\X);
		};

        \foreach \X in {e0,e0',e1,e1',e2,e2',e3,e3'}
		{
		\tkzDrawPoints[mydot,size=6,color=dark-green,thick,fill=white](\X);
		};
		
		\begin{scope}[xshift=3ex]
		\tkzDefPoint(-0.1,0.3){gammaM};
		\tkzLabelPoint[orange](gammaM){\Large $\delta$}
		\tkzDefPoint(-1.6,0){gammaP};
		\tkzLabelPoint[blue](gammaP){\Large $\eta$}
		\tkzDefPoint(-1.2,1.0){deltav};
		\tkzLabelPoint[mypurple](deltav){\Large $\gamma_{(E_1)}$}
		\tkzDefPoint(-0.6,-0.4){deltaw};
		\tkzLabelPoint[mypurple](deltaw){\Large $\gamma_{(E_2)}$}
		\end{scope}
    \end{tikzpicture}
\caption{\label{fig:overlapextaccord} Illustration of an overlap extension.}
\end{figure}
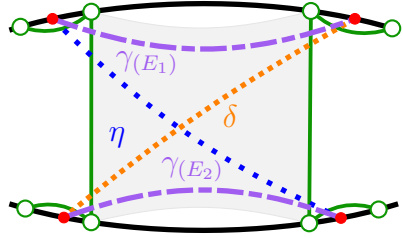 
    \item \new{Arrow extension} : 
    whenever we have a non-split short exact sequence  \[\begin{tikzcd}
	 \MM(\eta) & {E} & \MM(\delta) 
	\arrow[tail, from=1-1, to=1-2]
	\arrow[two heads, from=1-2, to=1-3]
\end{tikzcd},\] where $E \in \ind(Q,R)$, then we can construct $\gamma_{(E)}$ from $\delta$ and $\eta$ as described in \cref{fig:arrowextaccord}. Similarly to overlap extensions, by abuses of notations with strings of $(Q,R)$, we write $\ArExt(\delta,\eta)$ for the set of all the accordions $\gamma_{(E)}$ over all the equivalence classes of short exact sequences of the above shape.
 
    \begin{figure}[!ht]
\centering 
    \begin{tikzpicture}[mydot/.style={
					circle,
					thick,
					fill=white,
					draw,
					outer sep=0.5pt,
					inner sep=1pt
				}, scale = 1]
		\tikzset{
		osq/.style={
        rectangle,
        thick,
        fill=white,
        append after command={
            node [
                fit=(\tikzlastnode),
                orange,
                line width=0.3mm,
                inner sep=-\pgflinewidth,
                cross out,
                draw
            ] {}}}}
            
        \draw [line width=0.7mm,domain=-10:10] plot ({4*cos(\x)}, {1.5*sin(\x)});
		\draw [line width=0.7mm,domain=20:50] plot ({4*cos(\x)}, {1.5*sin(\x)});
		\draw [line width=0.7mm,domain=80:100] plot ({4*cos(\x)}, {1.5*sin(\x)});
		\draw [line width=0.7mm,domain=110:130] plot ({4*cos(\x)}, {1.5*sin(\x)});
		\draw [line width=0.7mm,domain=170:190] plot ({4*cos(\x)}, {1.5*sin(\x)});
        \draw [line width=0.7mm,domain=250:290] plot ({4*cos(\x)}, {1.5*sin(\x)});
		\foreach \X in {0,1}
		{
		\tkzDefPoint(4*cos(pi/2*(\X-0.3) +pi/3),1.5*sin(pi/2*(\X-0.3) + pi/3)){\X};
		};
		\foreach \X in {2,3}
		{
		\tkzDefPoint(4*cos(pi/6*(\X-3) +3*pi/2),1.5*sin(pi/6*(\X-3) + 3*pi/2)){\X};
		};
		
		\tkzDefPoint(4,0){4};
		\tkzDefPoint(0,1.5){5};
		\tkzDefPoint(-4,0){6};
		\tkzDefPoint(4*cos(3*pi/2-pi/11),1.5*sin(3*pi/2-pi/11)){7};
		\tkzDefPoint(4*cos(3*pi/2+pi/11),1.5*sin(3*pi/2+pi/11)){8};
		
		\filldraw[gray,opacity=0.1] (4) to [bend left=20] (5) to [bend left=20] (6) to [bend left=10] (7) to [bend right=10] (8) to [bend left=10] cycle;
		
		\draw[line width=0.5mm,bend left = 20,dark-green] (4) to (5) to (6);
		\draw[line width=0.5mm,bend left = 10,dark-green,dashed] (6) to (7);
		\draw[line width=0.5mm,bend left = 10,dark-green,dashed] (8) to (4);
		
		\draw[line width=0.7mm ,bend right=10,orange, dotted](0) edge (3);
		
		\draw[line width=0.7mm ,bend left=10,blue, loosely dotted](1) edge (3);
		
		\draw [line width=0.7mm, mypurple,dash pattern={on 10pt off 2pt on 5pt off 2pt}, bend left=20] (0) edge (1);

		\foreach \X in {0,1,3}
		{
		\tkzDrawPoints[circle,fill =red,size=4,color=red](\X);
		};
		
		\foreach \X in {4,...,8}
		{
		\tkzDrawPoints[mydot,size=6,color=dark-green,thick,fill=white](\X);
		};

		\tkzDefPoint(0.9,0.1){gammaM};
		\tkzLabelPoint[orange](gammaM){\Large $\delta$}
		\tkzDefPoint(-1.5,0.3){gammaP};
		\tkzLabelPoint[blue](gammaP){\Large $\eta$}
		\tkzDefPoint(-0.1,1.2){deltav};
		\tkzLabelPoint[mypurple](deltav){\Large $\gamma_{(E)}$}
    \end{tikzpicture}
\caption{\label{fig:arrowextaccord} Illustration of an arrow extension. The gray area corresponds to one cell of $\pmb{\Gamma}(\Delta^{\gpoint})$.}
\end{figure}
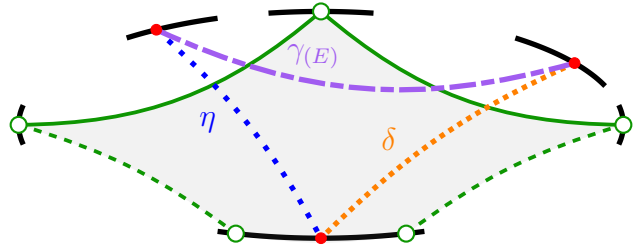
\end{enumerate}
\end{prop}

\subsection{Neighboring projective accordions}
\label{ss:neighboringproj} We will use the geometric model to compute resolving subcategories and consider minimal projective resolutions.
Let $(Q,R)$ be a gentle quiver with finite global dimension. This is equivalent to $\mathcal{M} \subset \partial \pmb{\Sigma}$. The $\rpoint$-dissection $\Prj(\Delta^{\gpoint})$ given by the projective accordions is homeomorphic to $\Delta^{\gpoint}$. To simplify the computations for resolving subcategories, we choose to work with $(\pmb{\Sigma}, \mathcal{M}, \Prj(\Delta^{\gpoint}))$ rather than $\Surf(Q,R)$. 

As $\mathcal{M}_{\rpoint} \subset \partial \pmb{\Sigma}$, for any $v \in \mathcal{M}_{\rpoint}$, write $\Accord_{[e]}$ for the set of accordions  admitting $e$ as an endpoint. Later on, for any $\mathfrak{B} \subseteq \Accord$, we set $\mathfrak{B}_{[e]} = \Accord_{[e]} \cap \mathfrak{B}$. We can consider a total order $\preccurlyeq_v$ on $\Accord_{[v]}$ as follows. For $(\delta,\eta) \in \Accord_{[v]}^2$, we write $\eta \preccurlyeq_v \delta$ whenever $\eta$ follows $\delta$ with respect to the counter-clockwise orientation around $v$.

Given a non-projective accordion $\delta$, we also defined the \new{neighboring projective accordions} of $\delta$ to be the projective accordions $\rho$ such that either $\MM(\rho)$ appears as a summand of a projective representation in the minimal projective resolution of $\MM(\delta)$, or there exists $i>0$ such that  $\Ext^i(\MM(\delta),\MM(\rho))\neq0$. Denote by $\NP(\delta)$ the set of neighboring projective accordions of $\delta$. We can characterize them combinatorially.

\begin{prop}[\cite{DS251}]\label{prop:CombidescripNproj}
Let $(Q,R)$ be a representation-finite gentle quiver. Let $\delta$ be an accordion of $\Surf(Q,R)$. Then $\NP(\delta)$ is the subset of $\Prj(\Delta^{\gpoint})$ made of all accordions $\eta$ satisfying at least one of the following conditions:
\begin{enumerate}[label=$(\roman*)$,itemsep=0.2em]
    \item the curve $\eta$ crosses $\delta$: they correspond to projective accordions such that $\OvExt(\delta,\eta) \neq \varnothing$;
    \item the curve $\eta$ is part of the border of a cell crossed by $\delta$ containing one of its endpoints and is a part of the path of projective curves that links the last projective crossed to the endpoint of $\delta$: they correspond to projective accordions such that $\MM(\eta)$ is a summand of a projective representation in the minimal projective resolution of $\MM(\delta)$; or,
    \item the curve $\eta$ and $\delta$ have a common endpoint $v$, and, at the vertex $v$, all the projective accordions are smaller than $\delta$ with respect to $\preccurlyeq_v$: they are projective accordions such that either $\ArExt(\delta,\eta) \neq \varnothing$, or $\Ext^i(\MM(\delta), \MM(\eta)) \neq 0$, for some $i > 1$.
\end{enumerate}
\end{prop}

We also recall a result that allows one to characterize $\rpoint$-arcs that are accordions of $(\pmb{\Sigma}, \mathcal{M}, \Delta^{\gpoint})$ directly via the projective dissection $\Prj(\Delta^{\gpoint})$.

Let us introduce a notation. Given $C \in \pmb{\Gamma}(\Prj(\Delta^{\gpoint}))$ a cell of the projective dissection, we can order $\rpoint$-points in $\partial C$ following the boundary of C in counter-clockwise
order from the unique $\gpoint \in \partial C$. We denote by $m_C$ the last $\rpoint$-point in $ \partial C$ under this ordering. 

\begin{prop}[\cite{DS251}] \label{prop:arcsaccordions} Let $(\pmb{\Sigma}, \mathcal{M}, \Delta^{\gpoint})$ be a $\gpoint$-dissected marked surface such that $\mathcal{M} \subset \partial \pmb{\Sigma}$. A $\rpoint$-arc $\delta$ is an accordion if and only if either $\delta \in \Prj(\Delta^{\gpoint})$ or the following assertions hold: 
\begin{enumerate}[label=$(\alph*)$, itemsep=1mm]
    \item whenever $\delta$ enters a cell $C$ of $\Prj(\Delta^{\gpoint})$ by crossing an $\rpoint$-arc $\mu$,
    \begin{enumerate}[label=$(a \arabic*)$, itemsep=1mm]
        \item if it leaves $C$, it leaves it by crossing an $\rpoint$-arc $\nu$ adjacent to $\mu$;
        \item otherwise, its endpoint in $\partial C$ is either $m_C$ or the non-common endpoint of an $\rpoint$-arc adjacent to $\mu$, or 
    \end{enumerate}
    \item if $\delta$ is contained in a cell $C$, then:
    \begin{enumerate}[label=$(b \arabic*)$, itemsep=1mm]
        \item its endpoints are the distinct endpoints of a pair of adjacent projective accordions; or,
        \item one of its endpoints must be $m_C$. 
    \end{enumerate}
\end{enumerate}
\end{prop}

\subsection{Tilting dissections}
\label{ss:rigandtiltdissec} In this section, we recall the main results from \cite{chang2025}. It allows us to describe subsets of $\Accord$  that correspond to rigid or tilting objects in $\rep(Q,R)$ via the geometric model.

We recall that a gentle quiver $(Q,R)$ is said to be \emph{representation-finite} whenever $\#\ind(Q,R) < \infty$. Given a collection $\mathfrak{U} \subset \ind(Q,R)$, we set:
\begin{enumerate}[label=$\bullet$, itemsep=1mm]
    \item $\mathfrak{U}^{\perp_{>0}}  = \ind\left( X \mid \forall i > 0,\ \forall M  \in \mathfrak{U},\ \Ext^i(M,X) = 0 \right)$; and,
    \item ${}^{\perp_{>0}}\mathfrak{U} = \ind\left( X \mid \forall i > 0,\ \forall M  \in \mathfrak{U},\ \Ext^i(X,M)=0 \right)$.
\end{enumerate}
An indecomposable representation $N \in \ind(Q,R)$ is said to be \new{$\Ext$-orthogonal} to $\mathfrak{U}$ whenever $N \in  \mathfrak{U}^{\perp_{>0}} \cap {}^{\perp_{>0}}\mathfrak{U}$. In the following $M,N \in \rep(Q,R)$, we say that $M$ and $N$ are $\Ext$-orthogonal, and we write $M \Extperp N$, whenever $M$ is $\Ext$-orthogonal to $\add(N)$, or, equivalently, whenever $N$ is $\Ext$-orthogonal to $\add(M)$. We extend those notations to accordions, and a subset of accordions on $\Surf(Q,R) = (\pmb{\Sigma}, \mathcal{M}, \Delta^{\gpoint})$.

\begin{definition}[\cite{HAHK07,HR82}] \label{def:rigidtilting}
Let $(Q,R)$ be a representation-finite gentle quiver with finite global dimension. A collection $\mathfrak{T} \subseteq \ind(Q,R)$ is said to be \new{rigid} whenever, for all $(T_1,T_2) \in \mathfrak{T}^2$, we have $T_1 \Extperp T_2$. A \new{tilting} collection $\mathfrak{T}$ is a rigid collection of $\#Q_0$ nonisomorphic indecomposable representations. We denote by $\Tilt(Q,R)$ the set of tilting collections in $\ind(Q,R)$.
\end{definition}

\begin{remark} \label{rem:rigicolancrigidobj} If a collection $\mathfrak{T}$ is rigid, then $\oplus_{T \in \mathfrak{T}} T \in \rep(Q,R)$ is rigid is the common sense. Note that we do not need to ask for finite projective dimension, as $(Q,R)$ is assumed to have finite global dimension. An equivalent statement holds for tilting collections. 
\end{remark}

For any $\mathfrak{B} \subseteq \ind(Q,R)$, we denote by $\bbDelta_{(\mathfrak{B})}$ to be the set of accordions $\gamma_{(M)}$ for $M \in \mathfrak{B}$. We extend this notation to any subcategory $\mathscr{C} \subseteq \rep(Q,R)$, by taking the indecomposable objects that generate $\mathscr{C}$. For any object $M \in \rep(Q,R)$, we simply write $\bbDelta_{(M)}$ for $\bbDelta_{(\add(M))}$.

\begin{definition}
\label{def:tiltdissec}
    Let $(Q,R)$ be a representation-finite gentle quiver with finite global dimension. Set $\Surf(Q,R) = (\pmb{\Sigma}, \mathcal{M}, \Delta^{\gpoint})$. A subset $\mathfrak{D} \subset \Accord$ is \new{rigid} if there exists a rigid collection $\mathfrak{T}$ such that $\mathfrak{D} = \bbDelta_{(\mathfrak{T})}$. We say that $\mathfrak{D}$ is \new{tilting} if we can choose a tilting collection $\mathfrak{T}$.
\end{definition}

Because of the overlap extensions obtained when accordions cross, it follows that if $\mathfrak{D}$ is rigid, then $\mathfrak{D}$ must be a $\rpoint$-dissection. By \cite{APS19}, we can indeed prove that $\mathfrak{D}$ must be an admissible $\rpoint$-dissection. Moreover, if $\mathfrak{D}$ is tilting, then it must be a dualizable $\rpoint$-dissection. However, the converse is false. We must take into account additional parameters, called \emph{weights}, defined as follows.

\begin{definition}
    \label{def:weightingcorner}
    Let $e \in \mathcal{M}_{\rpoint}$. Denote by $C \in \pmb{\Gamma}(\Delta^{\gpoint})$ the unique $\gpoint$-cell such that $e \in \partial C$. For any $\delta,\eta \in \Accord_{[e]}$, we set $\wt_e(\delta,\eta)$ to be the number of points in $\mathcal{M}_{\gpoint}$ appearing in the boundary of the cell whose delimited by $\partial C$, $\delta$, and $\eta$. We call it the \new{weight} of $(\delta,\eta)$ relatively to $e$.
\end{definition}

In \cref{fig:exwtcalc}, we draw a general example of two accordions $\delta,\eta \in \Accord$ sharing a common endpoint $e \in \mathcal{M}_{\rpoint}$  such that $\wt_e(\delta,\eta) = 3$.

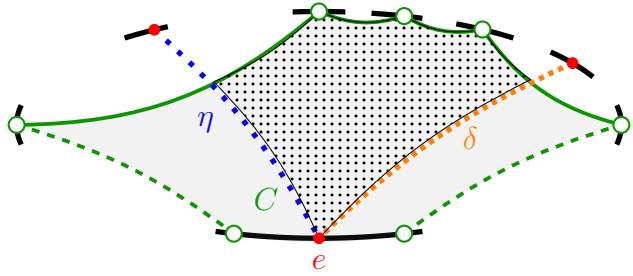
\begin{figure}[!ht]
\centering 
    \begin{tikzpicture}[mydot/.style={
					circle,
					thick,
					fill=white,
					draw,
					outer sep=0.5pt,
					inner sep=1pt
				}, scale = 1]
		\tikzset{
		osq/.style={
        rectangle,
        thick,
        fill=white,
        append after command={
            node [
                fit=(\tikzlastnode),
                orange,
                line width=0.3mm,
                inner sep=-\pgflinewidth,
                cross out,
                draw
            ] {}}}}
            
        \draw [line width=0.7mm,domain=-10:10] plot ({4*cos(\x)}, {1.5*sin(\x)});
		\draw [line width=0.7mm,domain=25:40] plot ({4*cos(\x)}, {1.5*sin(\x)});
        \draw [line width=0.7mm,domain=50:63] plot ({4*cos(\x)}, {1.5*sin(\x)});
        \draw [line width=0.7mm,domain=70:80] plot ({4*cos(\x)}, {1.5*sin(\x)});
		\draw [line width=0.7mm,domain=85:95] plot ({4*cos(\x)}, {1.5*sin(\x)});
		\draw [line width=0.7mm,domain=120:130] plot ({4*cos(\x)}, {1.5*sin(\x)});
		\draw [line width=0.7mm,domain=170:190] plot ({4*cos(\x)}, {1.5*sin(\x)});
        \draw [line width=0.7mm,domain=250:290] plot ({4*cos(\x)}, {1.5*sin(\x)});
		\foreach \X in {0,1}
		{
		\tkzDefPoint(4*cos(pi/2*(\X-0.3) +pi/3),1.5*sin(pi/2*(\X-0.3) + pi/3)){\X};
		};
		\foreach \X in {2,3}
		{
		\tkzDefPoint(4*cos(pi/6*(\X-3) +3*pi/2),1.5*sin(pi/6*(\X-3) + 3*pi/2)){\X};
		};
		
		\tkzDefPoint(4,0){4};
		\tkzDefPoint(0,1.5){5};
		\tkzDefPoint(-4,0){6};
		\tkzDefPoint(4*cos(3*pi/2-pi/11),1.5*sin(3*pi/2-pi/11)){7};
		\tkzDefPoint(4*cos(3*pi/2+pi/11),1.5*sin(3*pi/2+pi/11)){8};
        \tkzDefPoint(4*cos(pi/2-pi/11),1.5*sin(pi/2-pi/11)){9};
		\tkzDefPoint(4*cos(pi/2-2*pi/11),1.5*sin(pi/2-2*pi/11)){10};
		
		\filldraw[gray,opacity=0.1] (4) to [bend left=20] (10) to [bend left=20] (9) to [bend left=20] (5) to [bend left=20] (6) to [bend left=10] (7)  to [bend right=10] (8) to [bend left=10] cycle;

		\draw[line width=0.5mm,bend left = 20,dark-green] (4) to (10) to (9) to (5) to (6);
		\draw[line width=0.5mm,bend left = 10,dark-green,dashed] (6) to (7);
		\draw[line width=0.5mm,bend left = 10,dark-green,dashed] (8) to (4);
		
		\draw[line width=0.7mm ,bend right=10,orange, dotted](0) edge (3);
		
		\draw[line width=0.7mm ,bend left=10,blue, loosely dotted](1) edge (3);

        \filldraw [pattern=dots, pattern color=black,line width=0mm] (10) to [bend left=20] (9) to [bend left=20] (5) to [bend left=10] (-1.39,0.57) to [bend left=10] (3) to [bend left=10] (2.8,.6) to [bend left = 10] cycle;

		\foreach \X in {0,1,3}
		{
		\tkzDrawPoints[circle,fill =red,size=4,color=red](\X);
		};
		
		\foreach \X in {4,...,10}
		{
		\tkzDrawPoints[mydot,size=6,color=dark-green,thick,fill=white](\X);
		};

		\tkzDefPoint(2,0.1){gammaM};
		\tkzLabelPoint[orange](gammaM){\Large $\delta$}
		\tkzDefPoint(-1.5,0.3){gammaP};
		\tkzLabelPoint[blue](gammaP){\Large $\eta$}
		\tkzDefPoint(-0.7,-.7){deltav};
		\tkzLabelPoint[dark-green](deltav){\Large $C$}
        \tkzDefPoint(0,-1.6){deltav};
		\tkzLabelPoint[red](deltav){\Large $e$}
    \end{tikzpicture}
    \caption{\label{fig:exwtcalc} In this example, we have $\wt_e(\delta,\eta) = 3$.}
\end{figure}

\begin{lemma}
    \label{lem:counterclockandwt}
    For any $e \in \mathcal{M}_{\rpoint}$, and for any $(\delta_1,\delta_2, \delta_3) \in \Accord_e^3$, if $\delta_1 \preccurlyeq_e \delta_2 \preccurlyeq_e \delta_3$, we have $\wt_e(\delta_1,\delta_2) + \wt_e(\delta_2,\delta_3) = \wt_e(\delta_1,\delta_3)$.
\end{lemma}

\begin{remark}
    \label{rem:wtnotdist}
    The map $\wt_e$ is not a distance on $\Accord_e$ in general.
\end{remark}

This weight encodes data for higher extensions between the associated indecomposable representations. 

\begin{prop}[\cite{chang2025}]
\label{prop:Extigeom}
Let $e \in \mathcal{M}_{\rpoint}$, and $\delta,\eta \in \Accord_{[e]}$. Assume that $\delta \prec_e \eta$. If $\wt_e(\delta,\eta) = i \neq 0$, then $\Ext^i(\MM(\delta), \MM(\eta)) \neq 0$ or $\Ext^i(\MM(\delta), \MM(\eta)) \neq 0$. Moreover, if $\wt_e(\delta,\eta) = 1$, then $\ArExt(\delta,\eta) \neq \varnothing$.
\end{prop}

We can therefore characterize combinatorially pairs of accordions that are $\Ext$-orthogonal.

\begin{prop}[\cite{chang2025}]
\label{prop:Extorthoaccord}
    Let $\delta,\eta \in \Accord$. Then $\delta \Extperp \eta$ if and only if the following assertions hold:
    \begin{enumerate}[label=$(\roman*)$, itemsep=1mm]
        \item $\delta$ and $\eta$ do not cross; and,
        \item for any common endpoint $e$, we have $\wt_e(\delta,\eta) = 0$.
    \end{enumerate}
\end{prop}

We can now provide a characterization of rigid and tilting subsets of accordions.

\begin{theorem}[\cite{chang2025}]
\label{thm:tiltgeom}
    Let $\mathfrak{D} \subseteq \Accord$. Then:
    \begin{enumerate}[label=$(\roman*)$, itemsep=1mm]
        \item $\mathfrak{D}$ is rigid if and only if $\mathfrak{D}$ is an admissible $\rpoint$-dissection of $(\pmb{\Sigma}, \mathcal{M})$ such that, for any $e \in \mathcal{M}_{\gpoint}$, and $(\delta, \eta) \in \mathfrak{D}_{[e]}$, we have $\wt_e(\delta,\eta) = 0$; and, 
        \item $\mathfrak{D}$ is tilting if and only if $\mathfrak{D}$ is rigid and dualizable.
    \end{enumerate}
\end{theorem}

\subsection{Gentle trees}
\label{ss:gentletrees}

In what follows, we restrict our study to \emph{gentle trees}. A \new{gentle tree} is a gentle quiver $(Q,R)$ such that $Q$ is an oriented tree. In such case, by setting $\Surf(Q,R) = (\pmb{\Sigma}, \mathcal{M}, \Delta^{\gpoint})$, we have that $\mathcal{M} \subset \partial \pmb{\Sigma}$, and $\pmb{\Sigma}$ is homeomorphic to a disc. In addition, we characterize indecomposable representations via their vertex support, which is equivalent to characterizing accordions on $\Surf(Q,R)$ by the $\gpoint$-arcs of $\Delta^{\gpoint}$ they cross. Furthermore, the space of homomorphisms between indecomposable representations is at most one-dimensional, which means that two arbitrary accordions can cross at most once up to homotopy.

\begin{prop}[\cite{BDMTY19,CPS21}]
Let $(Q,R)$ be a gentle tree. Up to isomorphism, all the extensions between indecomposable representations are either overlap extensions or arrow extensions.
\end{prop}

        \section{Resolving subcategories}
        \label{sec:ResSubcats}
        \pagestyle{plain}

In this section, we recall some results, notions, and notations introduced in \cite{DS251,DS252}.

\subsection{Generalities}
\label{ss:generalities}

Let $\mathscr{C}$ be a Krull--Schmidt abelian category with enough projective objects.

\begin{conv} \label{conv:add}
A subcategory is said to be additively closed if it is closed under direct sums and summands.
\end{conv}

\begin{definition} \label{def:resolv}
A full subcategory $\mathscr{R} \subseteq \mathscr{C}$ is called \new{resolving} if it satisfies the following conditions: 
\begin{enumerate}[label=$(\mathsf{R \arabic*})$, itemsep=1mm]
\setcounter{enumi}{-1}
\item \label{R0} $\mathscr{R}$ is additively closed
\item \label{R1} $\mathscr{R}$ generates $\mathscr{C}$ : every object of $\mathscr{C}$ admits a cover in $\mathscr{R}$,
\item \label{R2} $\mathscr{R}$ is closed under extensions, and,
\item \label{R3} $\mathscr{R}$ is closed under taking kernels of epimorphisms.
\end{enumerate}
\end{definition}

For any $i \geqslant 1$, and for any $M \in \mathscr{C}$, we denote by $\Omega^i(M)$ the \new{$i$th syzygy} of $M$. If $i = 1$, we denote it by $\Omega(M)$. 

\begin{lemma} \label{lem:othercharactresolv}
Let $\mathscr{R} \subseteq \mathscr{C}$ be a full additive subcategory. We have the following equivalences:
\begin{enumerate}[label=$(\roman*)$, itemsep=1mm]
\item \label{ires} $\mathscr{R}$ satisfies \ref{R1} if, and only if, $\proj(\mathscr{C}) \subseteq \mathscr{R}$;
\item $\mathscr{R}$ is resolving if, and only if,  $\mathscr{R}$ satisfies \ref{R0}, \ref{R1} \ref{R2}, and $\mathscr{R}$ is closed under taking syzygies.
\end{enumerate}
\end{lemma}

Let us prove a technical proposition that will be used later, in \cref{ss:projcompl}. For the moment, the result below highlights a particular family of resolving subcategories.

\begin{prop}
\label{prop:projorthgen}
    Let $\mathfrak{P}$ be a collection of indecomposable projective objects in $\mathscr{C}$. Then the subcategory $\mathscr{E}_\mathfrak{P}$ additively generated by the objects that are Ext-orthogonal to $\mathfrak{P}$ is resolving.
\end{prop}
\begin{proof}
By definition, $\mathscr{E}_\mathfrak{P}$ is additive, and  $\proj(\mathscr{C}) \subset \mathscr{E}_\mathfrak{P}$. We still must show that $\mathscr{E}_\mathfrak{P}$ is closed under extensions and kernels of epimorphisms.
 
Assume that we have the following short exact sequence
\[\begin{tikzcd}
	{X} & M & {Y},
	\arrow[tail, from=1-1, to=1-2]
	\arrow[two heads, from=1-2, to=1-3]
\end{tikzcd}\]
where $X,Y\in \mathscr{E}_\mathfrak{P}$. Consider $P \in \mathfrak{P}$. By applying the functor $\Hom(-,P)$, we get the following long exact sequence.
\[\begin{tikzcd}
 0 & \Hom(Y,P) & \Hom(M,P)  & \Hom(X,P) & \\
	  & \Ext^1(Y,P) & \Ext^1(M,P)  & \Ext^1(X,P) & \\
     & \Ext^2(Y,P) & \Ext^2(M,P) & \Ext^2(X,P) & ...
	\arrow[from=1-1, to=1-2]
	\arrow[from=1-2, to=1-3]
	\arrow[from=1-3, to=1-4]
	\arrow[from=1-4, to=2-2]
    \arrow[from=2-2, to=2-3]
    \arrow[from=2-3, to=2-4]
    \arrow[from=2-4, to=3-2]
    \arrow[from=3-2, to=3-3]
    \arrow[from=3-3, to=3-4]
    \arrow[from=3-4, to=3-5]
\end{tikzcd} \] 
By hypothesis, $\Ext^i(X,P) = 0 = \Ext^i(Y,P)$ for all $i > 1$ and $j \in \{1,2\}$. So $\Ext^i(M,P)=0$ for all $i > 1$. By dual arguments, we have that $\Ext^i(P,M) = 0$. Therefore, for all $P \in \mathfrak{P}$, we have $M \Extperp P$. So $M \in \mathscr{E}_{\mathfrak{P}}$.

Assume that we have the following short exact sequence
\[\begin{tikzcd}
	M & {X} & {Y},
	\arrow[tail, from=1-1, to=1-2]
	\arrow[two heads, from=1-2, to=1-3]
\end{tikzcd}\]
where $X,Y\in \mathscr{E}_\mathfrak{P}$. Consider $P \in \mathfrak{P}$. By applying the functor $\Hom(-,P)$, we get the following long exact sequence: 
\[\begin{tikzcd}
 0 & \Hom(Y,P) & \Hom(X,P)  & \Hom(M,P) & \\
	  & \Ext^1(Y,P) & \Ext^1(X,P)  & \Ext^1(M,P) & \\
     & \Ext^2(Y,P) & \Ext^2(X,P) & \Ext^2(M,P) & ...
	\arrow[from=1-1, to=1-2]
	\arrow[from=1-2, to=1-3]
	\arrow[from=1-3, to=1-4]
	\arrow[from=1-4, to=2-2]
    \arrow[from=2-2, to=2-3]
    \arrow[from=2-3, to=2-4]
    \arrow[from=2-4, to=3-2]
    \arrow[from=3-2, to=3-3]
    \arrow[from=3-3, to=3-4]
    \arrow[from=3-4, to=3-5]
\end{tikzcd} \]
As for extensions, we get that $\Ext^i(M,P) = 0$ for all $i > 0$, and by dual arguments, we have $M \Extperp P$. So $M \in \mathscr{E}_\mathfrak{P}$.
We got the desired result.
\end{proof}

We can simplify certain conditions to verify that a subcategory is resolving under the hypotheses we impose on $\ind(\mathscr{C})$ \cite[Theorem 2.11]{DS251}. So we can characterize resolving subcategories by their indecomposable non-projective objects. For any resolving subcategory $\mathscr{R} \subset \mathscr{C}$, we write $\pmb{\ind \setminus \proj}(\mathscr{R})$ for the set of indecomposable non-projective objects in $\mathscr{R}$ up to isomorphism. Moreover, the family of resolving subcategories is closed under arbitrary intersections. Therefore, we can define the \new{resolving closure} of any $\mathcal{X} \subseteq \pmb{\ind \setminus \proj}(\mathscr{C})$ to be the smallest resolving subcategory that contains $\mathcal{X}$. Denote it by $\Res(\mathcal{X})$. For $X,Y \in \ind(Q,R)$, we write $\Res(X)$ for $\Res(\{X\})$ and $\Res(X,Y)$ for $\Res(\{X,Y\})$. We say that a resolving subcategory $\mathscr{R}$ is \new{monogenous} whenever there exists $X \in \pmb{\ind \setminus \proj}(\mathscr{C})$ such that $\mathscr{R} = \Res(X)$.

We consider an analogous operator on subsets of $\pmb{\ind \setminus \proj}(\mathscr{C})$, denoted $(-)^{\Res}$, which returns $\pmb{\ind \setminus \proj}(\Res(\mathcal{X}))$ for any $\mathcal{X} \subseteq \pmb{\ind \setminus \proj}(\mathscr{C})$.

\subsection{Monogenous resolving subcategories}
\label{ss:monogenous}

Consider a gentle tree $(Q,R)$. We endow $\pmb{\ind \setminus \proj}(Q, R)$ with an order relation defined by \[\forall X,Y \in \pmb{\ind \setminus \proj}(Q,R),\ X \Resleq Y \Longleftrightarrow \Res(X) \subseteq \Res(Y).\]  We get an injective map from the finite complete lattice $(\ResOrd(Q,R), \subseteq)$ of resolving subcategories of $(Q,R)$ to the lattice of ideals of $(\pmb{\ind \setminus \proj}(Q,R), \Resleq)$. Moreover, we also have the following result.

\begin{theorem}[\cite{DS251}] \label{thm:Monoareallthejoinirred} The join-irreducible elements in $(\ResOrd(Q,R), \subseteq)$ are precisely given by the monogeneous resolving subcategories.
\end{theorem}

Recall the description of monogenous resolving subcategories via the geometric model. To do so, we use our variant of the geometric model (recall in \cref{ss:neighboringproj}) and first define a suitable coloring of $\NP(\delta)_0$ for a given $\delta \in \Accord'$.

\begin{definition}
\label{def:colourendpoints}
Let $(Q,R)$ be a gentle tree and $\Surf(Q,R) = (\pmb{\Sigma}, \mathcal{M}, \Delta^{\gpoint})$. Let $\delta \in \Accord$. We define a \new{coloration} of $\NP(\delta)_0$ to be a partition of $\NP(\delta)_0$ as  follows:
\begin{enumerate}[label=\arabic*),itemsep=1mm]
    \item  As $\delta$ cuts the disc into two connected parts, we consider that, on one side of the curve, points of $\NP(\delta)_0$ are colored red $\rsquare$, and on the other side they are colored green $\gsquare$. In what follows, we arbitrarily choose to consider the red dots above $\delta$ and the green ones below.
    \item The left endpoint of the curve is called the \emph{source} and is colored red $\rsquare$ if all the accordions in $\NP(\delta)$ sharing the same source extremity as $\delta$ admit a green endpoint $\gsquare$ at the other end. It is colored orange $\osquare$ otherwise.
    \item The right endpoint of the curve is called the \emph{target}, and is colored dually.
    \item In larger cells, the coloration is changed: all the intermediate points except the second to last are colored orange $\osquare$ 
    \item The endpoint of the border of the large cell next to the orange endpoint is finally colored in pink $\psquare$.
\end{enumerate}
We denote by:
\begin{enumerate}[label=$\bullet$,itemsep=1mm]
    \item $\NP(\delta)_{0}^{{\rsquare}}$ the set of points in $\NP(\delta)_0$ colored in red $\rsquare$;
    \item $\NP(\delta)_{0}^{{\osquare}}$ the set of points in $\NP(\delta)_0$ colored in orange $\osquare$;
    \item $\NP(\delta)_{0}^{{\gsquare}}$ the set of points in $\NP(\delta)_0$ colored in green $\gsquare$;
    \item $\NP(\delta)_{0}^{{\psquare}}$ the set of points in $\NP(\delta)_0$ colored in pink $\psquare$.
\end{enumerate}
We can encode the choice of such a coloration as a map $\col_\delta : \NP(\delta)_0 \longrightarrow \{{\rsquare},{\gsquare},{\osquare},{\psquare}\}$.
\end{definition}

\begin{lemma}[\cite{DS251}] \label{lem:uniqcolandlaterality}
Let $(Q,R)$ be a gentle tree and $\Surf(Q,R) = (\pmb{\Sigma}, \mathcal{M}, \Delta^{\gpoint})$. Let $\delta \in \Accord$. Then there exists a unique coloration of $\NP(\delta)_0$ up to exchanging the ${\rsquare}$ and ${\gsquare}$ points. 
\end{lemma}

In the following, if $\col_\delta : \NP(\delta)_0 \longrightarrow \{{\rsquare},{\gsquare},{\osquare},{\psquare}\}$ is a coloration of $\NP(\delta)_0$, we denote by $\overline{\col_\delta}$ the obtained coloration by exchanging the ${\rsquare}$ and ${\gsquare}$ points, and call it its \new{conjugate coloration}.

\begin{remark} \label{rem:laterality}
A coloration of $\NP(\delta)_0$ induces a laterality of the surface by fixing a top and a bottom and thus a left and a right. Any non-projective accordion $\varsigma$ crosses or shares an endpoint with a projective accordion that has at least one $\rsquare,\gsquare$ or $\psquare$ endpoint. This endpoint is considered either on top of or at the bottom of $\varsigma$, depending on the orientation determined by the coloration of $\NP(\delta)_0$. In the following, we call the source of $\varsigma$ the left endpoint of $\varsigma$ following this new orientation and denote it by  $s(\varsigma)$. The other endpoint is called the target and denoted by $t(\varsigma)$
\end{remark}

Given $\delta \in \Accord'$, this coloration allows us to define a precise collection of accordions, namely the one corresponding to the indecomposable representations that additively generate $\Res(\MM(\delta))$.

\begin{definition} \label{def:geometricres}
Let $(Q,R)$ be a gentle tree, and $\Surf(Q,R) = (\pmb{\Sigma}, \mathcal{M}, \Delta^{\gpoint})$. For any $\delta \in \Accord$, we define the \new{monogeneous geometric resolving set} of $\delta$ as the set $\opResAc(\delta) = \opResAc'(\delta) \cup \Prj(\Delta^{\gpoint})$ where: \[\opResAc'(\delta) = \{ \eta \in \Accord \mid s(\eta) \in \NP(\delta)_0^{{\rsquare}} \cup \NP(\delta)_0^{{\osquare}}, t(\eta) \in \NP(\delta)_0^{{\gsquare}} \cup \NP(\delta)_0^{{\osquare}} \}.\]
The \new{monogeneous geometric resolving subcategory} of $\delta$, denoted by \new{$\mathscr{U}_\delta$}, is the additive subcategory of $\rep(Q,R)$ generated by $\{\MM(\varsigma) \mid \varsigma \in \opResAc(\delta)\}$.
\end{definition}

Let $\delta \in \Accord'$. We recall a technical lemma characterizing $\rpoint$-arcs $\varsigma$ such that $s(\varsigma) \in \NP(\delta)_0^{{\rsquare}} \cup \NP(\delta)_0^{{\osquare}}$ and $t(\varsigma) \in \NP(\delta)_0^{{\osquare}} \cup \NP(\delta)_0^{{\gsquare}}$ that are accordions of $(\pmb{\Sigma}, \mathcal{M}, \Delta^{\gpoint})$.

For any $\delta \in \Accord'$ which is not contained in a cell of the $\rpoint$-dissection $(\pmb{\Sigma}, \mathcal{M}, \Prj(\Delta^{\gpoint}))$, define $\tc (\delta)$ to be the pair $(\eta_R, C_R)$ where $\eta_R\in \Prj(\Delta^{\gpoint})$, and $C_R \in \pmb{\Gamma}(\Prj(\Delta^{\gpoint}))$ admitting $\eta_R$ as one of its edges, such that $\delta$ crosses $\eta_R$ and $t(\delta)$ is a vertex of $C_R$. If $\delta \subset C$ for some $\pmb{\Gamma}(\Prj(\Delta^{\gpoint}))$, then we set $\tc(\delta) = (\varnothing,C)$. We define $\sc(\delta)$ dually.

Given a $C \in \pmb{\Gamma}(\Prj(\Delta^{\gpoint}))$, we denote by $C(\delta)_0^{{\osquare}}$ the set of vertices of $C$ in $\NP(\delta)_0^{{\osquare}}$. We define similarly $C(\delta)_0^{{\gsquare}}$, $C(\delta)_0^{{\rpoint}}$ and $C(\delta)_0^{{\psquare}}$.

Define $\vartheta_1^R, \vartheta_2^R \in \Prj(\Delta^{\gpoint})$ such that:
\begin{enumerate}[label=$\bullet$,itemsep=1mm]
    \item $t(\vartheta_1^R) = t(\delta)$ and $\vartheta_1^R$ covers $\delta$ with respect to the counterclockwise order around $t(\delta)$ on $\Prj(\Delta^{\gpoint}) \cup \{\delta\}$; and,
    \item $t(\vartheta_2^R) = s(\vartheta_1^R)$ and $\vartheta_2^R$ covers $\vartheta_1^R$ with respect to the counterclockwise order around $s(\vartheta_1^R)$ on $\Prj(\Delta^{\gpoint}) \cup \{\delta\}$.
\end{enumerate}
We define $w_R(\delta) = s(\vartheta_2^R)$. We define $w_L(\delta)$ dually. See \cref{fig:exsctc} to visualize the different objects introduced.

\begin{figure}[!ht]
    \centering
    \begin{tikzpicture}[mydot/.style={
					circle,
					thick,
					fill=white,
					draw,
					outer sep=0.5pt,
					inner sep=1pt
				}, scale = 1]
		\tikzset{
		osq/.style={
        rectangle,
        thick,
        fill=white,
        append after command={
            node [
                fit=(\tikzlastnode),
                orange,
                line width=0.3mm,
                inner sep=-\pgflinewidth,
                cross out,
                draw
            ] {}}}}
        \draw [line width=0.7mm,domain=10:22] plot ({5*cos(\x)}, {2*sin(\x)});
		\draw [line width=0.7mm,domain=25:70] plot ({5*cos(\x)}, {2*sin(\x)});
		\draw [line width=0.7mm,domain=78:85] plot ({5*cos(\x)}, {2*sin(\x)});
		\draw [line width=0.7mm,domain=95:102] plot ({5*cos(\x)}, {2*sin(\x)});
		\draw [line width=0.7mm,domain=110:170] plot ({5*cos(\x)}, {2*sin(\x)});
		\draw [line width=0.7mm,domain=225:235] plot ({5*cos(\x)}, {2*sin(\x)});
		\draw [line width=0.7mm,domain=250:283] plot ({5*cos(\x)}, {2*sin(\x)});
		\draw [line width=0.7mm,domain=300:355] plot ({5*cos(\x)}, {2*sin(\x)});
		\foreach \X in {0,...,43}
		{
		\tkzDefPoint(5*cos(pi/22*\X),2*sin(pi/22*\X)){\X};
		};

		\draw[line width=0.7mm,blue, loosely dashed](4) to[bend left=30] (2,0) to[bend right=30] (-1,-1) to [bend right=30] (31);
		
		\draw[line width=0.9mm ,bend right=30,red, densely dashdotted](4) edge (43);
		\draw[line width=0.9mm ,bend right=30,red, densely dashdotted](4) edge (2);
		\draw[line width=0.9mm ,bend right=30,red, densely dashdotted](43) edge (41);
		\draw[line width=0.9mm ,bend right=30,red, densely dashdotted](41) edge (39);
		\draw[line width=0.9mm ,bend right=30,red, densely dashdotted](39) edge (37);
		\draw[line width=1.5mm ,bend left=30,red, densely dashdotted](37) edge (8);
		\draw[line width=0.9mm ,bend left=40,red, densely dashdotted](37) edge (10);
		\draw[line width=0.9mm ,bend left=40,red, densely dashdotted](34) edge (12);
		\draw[line width=1.5mm ,bend left=30,red, densely dashdotted](34) edge (14);
		\draw[line width=0.9mm ,bend right=30,red, densely dashdotted](16) edge (14);
		\draw[line width=0.9mm ,bend right=30,red, densely dashdotted](18) edge (16);
		\draw[line width=0.9mm ,bend right=30,red, densely dashdotted](20) edge (18);
		\draw[line width=0.9mm ,bend right=30,red, densely dashdotted](31) edge (20);
		\draw[line width=0.9mm ,bend right=30,red, densely dashdotted](31) edge (28);

		\filldraw [fill=red,opacity=0.1] (4) to [bend right=30] (43) to [bend right=30] (41) to [bend right=30] (39) to [bend right=30] (37) to [bend left=30] (8) to [bend left=10] cycle ;
		
		\filldraw [fill=red,opacity=0.1] (34) to [bend left=30] (14) to [bend left=30] (16) to [bend left=30] (18) to [bend left=30] (20) to [bend left=30] (31) to [bend right=10] cycle ;
		
		\foreach \X in {8,10,12,28}
		{
		\tkzDrawPoints[rectangle,fill =red,size=6,color=red](\X);
		};
		
		\foreach \X in {2,34}
		{
		\tkzDrawPoints[rectangle,size=6,color=dark-green,thick,fill=white](\X);
		};
		\foreach \X in {4,14,16,18,31,37,39,41}
		{
		\tkzDrawPoints[size=6,orange,osq](\X);
		};
		\foreach \X in {20,43}
		{
		\tkzDrawPoints[size=6,darkpink,line width=0.5mm,cross out, draw](\X);
		};

		\tkzDefPoint(-3.6,1){gammaM};
		\tkzLabelPoint[red](gammaM){\Large $C_L$}
		\tkzDefPoint(-1.4,1){gammaP};
		\tkzLabelPoint[red](gammaP){\Large $\pmb{\eta_L}$}
		\tkzDefPoint(1.4,1){v};
		\tkzLabelPoint[red](v){\Large $\pmb{\eta_R}$}
		\tkzDefPoint(3.7,-0.3){w};
		\tkzLabelPoint[red](w){\Large $C_R$}
		\tkzDefPoint(0,-0.2){deltav};
		\tkzLabelPoint[blue](deltav){\Large $\delta$}
		\tkzDefPoint(-4.7,1.8){deltaw};
		\tkzLabelPoint[orange](deltaw){$w_L(\delta)$}
		\tkzDefPoint(4.7,-0.9){w};
		\tkzLabelPoint[orange](w){$w_R(\delta)$}
    \end{tikzpicture}
    \caption{\label{fig:exsctc} Illustration of  $\sc(\delta) = (\eta_L, C_L)$ and $\tc(\delta) = (\eta_R, C_R)$ given $\delta \in \Accord$.}
\end{figure}
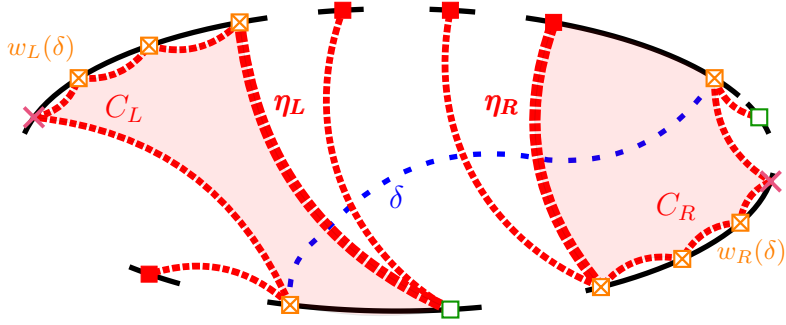

\begin{lemma}[\cite{DS251}] \label{lem:conditionsResAc}
Let $(\pmb{\Sigma}, \mathcal{M}, \Delta^{\gpoint})$ be a $\gpoint$-dissected marked disc with $\mathcal{M} \subset \partial \pmb{\Sigma}$. Let $\delta \in \Accord'$. If $\delta \nsubseteq C$ for some $C \in \pmb{\Gamma}(\Prj(\Delta^{\gpoint}))$, then we set $\sc(\delta) = (\eta_L,C_L)$ and $\tc(\delta) = (\eta_R,C_R)$; otherwise, we set $C=C_L=C_R$ the cell containing $\delta$. The curve $\varsigma \in \opResAc'(\delta)$ must satisfy all of the following assertions: 
\begin{enumerate}[label=$(\roman*)$, itemsep=1mm]
        \item \label{1Accord} if $\varsigma$ crosses $\eta_L$ and $w_L(\delta) \neq t(\eta_L)$, then $s(\varsigma) = s(\delta)$, and, dually, if $\varsigma$ crosses $\eta_R$ and $w_R(\delta) \neq s(\eta_R)$, then $t(\varsigma) = t(\delta)$;
        \item \label{2Accord} if $\varsigma \subset C_L$, then $s(\varsigma)=s(\delta)$, and, dually, if $\varsigma \subset C_R$, then $t(\varsigma) = t(\delta)$;
        \item \label{3Accord} if $s(\varsigma)\in \NP(\delta)_0^{{\rsquare}}$ and $t(\varsigma) \in C_R(\delta)_0$, then $t(\varsigma)=w_L(\delta)$, and, dually, if $t(\varsigma)\in \NP(\delta)_0^{{\gsquare}}$ and $s(\varsigma) \in C_R(\delta)_0$, then $s(\varsigma)=w_R(\delta)$;
        \item \label{4Accord} $s(\varsigma) \notin C_L(\delta)_0^{{\osquare}}\setminus \{ s(\delta),t(\eta_L)\}$, and $t(\varsigma) \notin C_R(\delta)_0^{{\osquare}}\setminus \{ t(\delta),s(\eta_R)\}$.
    \end{enumerate}
\end{lemma}

Following further technical calculations, we obtain a geometric description of the monogeneous resolving subcategories of $\rep(Q,R)$ in terms of the monogeneous geometric resolving subcategories.

\begin{theorem}[\cite{DS251}]
\label{thm:res_clo_1}
Let $(Q,R)$ be a gentle tree, and $\Surf(Q,R) = (\pmb{\Sigma}, \mathcal{M}, \Delta^{\gpoint})$. For any $\delta \in \Accord$, we have that \[\Res(\MM(\delta)) = \add\left( \MM(\eta) \mid \eta \in \opResAc(\delta) \right) = \mathscr{U}_\delta.\]
\end{theorem}

\subsection{Closure operators}
\label{ss:closure}
For any given set $\mathfrak{A}$, we define a \new{closure operator} on $\mathfrak{A}$ as a map $\cl : \mathcal{P}(\mathfrak{A}) \rightarrow \mathcal{P}(\mathfrak{A})$ such that:
\begin{enumerate}[label=$\bullet$, itemsep=1mm]
    \item for all $\mathfrak{B} \in  \mathcal{P}(\mathfrak{A})$, $\mathfrak{B} \subseteq \cl(\mathfrak{B})$,
    \item for all $\mathfrak{B} \in  \mathcal{P}(\mathfrak{A})$, $\cl(\mathfrak{B}) = \cl(\cl(\mathfrak{B}))$, and,
    \item for all $\mathfrak{B}, \mathfrak{C} \in \mathcal{P}(\mathfrak{A})$, if $\mathfrak{B} \subseteq \mathfrak{C}$, then $\cl(\mathfrak{B}) \subseteq \cl(\mathfrak{C})$.
\end{enumerate}
A subset $\mathfrak{B} \in \mathcal{P}(\mathfrak{A})$ is \new{$\cl$-closed} if $\mathfrak{B} = \cl(\mathfrak{B})$.
We denote by $\mathcal{P}_{\cl}(\mathfrak{A})$ the subset of $\mathcal{P}(\mathfrak{A})$ made of $\cl$-closed sets.

In our setting, the operators $(-)^{\Res}$ and $\opResAc'$ are closure operators on respectively $\pmb{\ind \setminus \proj}(Q,R)$ and $\Accord'$. As a result, we proved the following result.

\begin{prop}[\cite{DS252}] \label{prop:folkloreres}
Let $(Q,R)$ be a gentle tree. The following assertions hold:
\begin{enumerate}[label=$(\roman*)$,itemsep=1mm]
    \item The poset $(\mathcal{P}_{\Res}(\pmb{\ind \setminus \proj}(Q,R)), \subseteq)$ is a subposet of\\ $(\mathscr{J}(\pmb{\ind \setminus \proj}(Q,R), \Resleq), \subseteq)$;
    \item The lattice $(\mathcal{P}_{\Res}(\pmb{\ind \setminus \proj}(Q,R)), \subseteq, \cap, \veebar)$  is isomorphic to \\ $(\ResOrd(Q,R), \subseteq, \cap, \underline{\cup})$;
    \item The join-irreducible elements of $(\mathcal{P}_{\Res}(\pmb{\ind \setminus \proj}(Q,R)), \subseteq, \cap, \underline{\cup})$ are the principal ideals of $(\pmb{\ind \setminus \proj}(Q,R), \Resleq)$.
\end{enumerate}
\end{prop}

\begin{prop}[\cite{DS252}] \label{prop:folkloreresaccord}
The following assertions hold:
\begin{enumerate}[label=$(\roman*)$,itemsep=1mm]
    \item The operator $\opResAc'$ is a closure operator on $\Accord'$;
    \item The poset $(\mathcal{P}_{\opResAc'}(\Accord'), \subseteq)$ is a subposet of $(\mathscr{J}(\Accord', \Accordleq), \subseteq)$;
    \item The lattice $(\mathcal{P}_{\opResAc'}(\Accord'), \subseteq, \cap, \veebar)$, where \[\forall (\mathfrak{B}, \mathfrak{C}) \in \mathcal{P}_{\opResAc'}(\Accord')^2,\ \mathfrak{B} \veebar \mathfrak{C} = \opResAc'(\mathfrak{B} \cup \mathfrak{C}),\] is isomorphic to $(\ResOrd(Q,R), \subseteq, \cap, \underline{\cup})$;
    \item The join-irreducible elements of $(\mathcal{P}_{\opResAc'}(\Accord'), \subseteq, \cap, \underline{\cup})$ are the principal ideals of $(\Accord', \Accordleq)$;
    \item All the ideals of $(\Accord', \Accordleq)$ are closed under non-projective syzygies.
\end{enumerate}
\end{prop}

We recall that we can endow the antichains $\Anti(\Accord', \Accordleq)$ with a lattice structure inherited from $(\mathscr{J}(\Accord', \Accordleq), \subseteq)$. Set \[ \Theta : \left\{ \begin{matrix}
\Anti(\Accord', \Accordleq) & \longrightarrow & \mathscr{J}(\Accord',\Accordleq) \\
\mathfrak{B} & \longmapsto & \langle \mathfrak{B} \rangle
\end{matrix} \right.,\] where $\langle \mathfrak{B} \rangle$ is the ideal generated by $\mathfrak{B}$. Given $\mathfrak{B}, \mathfrak{C} \in \Anti(\Accord',\Accordleq)$, we say that $\mathfrak{B} \bleq \mathfrak{C}$ whenever $\Theta(\mathfrak{B}) \subseteq \Theta(\mathfrak{C})$. As $\Theta$ is a bijective map, $(\Anti(\Accord',\Accordleq), \bleq)$ is a poset, and we can define a join $\curlyvee$ and a meet $\curlywedge$ operation on this poset such that $\Theta$ becomes a lattice isomorphism. See \cite[Section 3.2]{DS252} for more details.

We adapted, therefore, \cite[Algorithm 3.17]{DS252} into an algorithm on the geometric model, to get the antichain $\mathfrak{B}$ in $(\Accord', \Accordleq)$ that generates, as an ideal, $\opResAc'(\mathcal{X})$. 

\begin{prop} \label{prop:maingeommotivation}
For any resolving subset $\mathfrak{U} \subseteq \mathcal{P}_{\opResAc'}(\Accord')$, there exists a unique antichain $\mathfrak{D}_\mathfrak{U} \in \Anti(\Accord', \Accordleq)$ such that \[\mathfrak{U} = \bigcup_{\delta \in \mathfrak{D}_{\mathfrak{U}}} \opResAc'(\delta). \]
Moreover $\mathfrak{D}_\mathfrak{U}$ can be obtained from any $\mathfrak{B} \subseteq \Accord'$ such that $\opResAc(\mathfrak{B}) = \mathfrak{U} $ by applying \cite[Algorithm 3.17]{DS252} adapted to the geometric model language.
\end{prop}

\subsection{Co-Z move}
\label{ss:CoZ}
In \cite{DS252}, we develop combinatorial techniques to adapt Algorithm 3.17 to the geometric model language, and determine completely with combinatorial operations, from any subset $\mathfrak{W} \subset \Accord'$, the antichain $\mathfrak{D}_\mathfrak{W}$ such that \[\opResAc'(\mathfrak{W}) = \bigcup_{\delta \in \mathfrak{D}_{\mathfrak{W}}} \opResAc'(\delta).\]
The primary operations, called \emph{moves}, applied to pairs of accordions in $\Accord'$, allow one to gradually reach this aim (see \cite[Algorithm 6.10]{DS252}. In this section, we recall essential settings about one of them: the \emph{$\CoZ$-move}. 

It records the behavior of two consecutive (higher) extensions involving a projective accordion $\rho \in \NP(\delta) \cap \NP(\eta)$ such that, by denoting $v$ and $w$ its extremities, we have  \[\col_\delta(\{v,w\}) \nsubseteq \{{\osquare}, {\psquare}\} \text{ and } \col_\eta(\{v,w\}) \nsubseteq \{{\osquare}, {\psquare}\}.\]
To define this move, we need to introduce vertices associated to $(\delta, \eta) \in (\Accord')^2$ and some given vertex $v \in \NP(\delta)_0 \cap \NP(\eta)_0$.

\begin{definition}
Let $\delta\in\Accord'$, $v\in\NP(\delta)_0^{{\gsquare}}$, $w\in\NP(\delta)_0^{{\rsquare}}$. The \new{upper source of} $\delta$ \new{associated to} $v$ denoted by $\s_{(\delta)}(v)$ as the source of the maximal accordion in  $\{\varsigma \mid \varsigma \in \opResAc'(\delta),\ t(\varsigma) = v\}$ with respect to $\Accordleq$.
Similarly, we define the \new{upper target of} $\delta$ \new{associated to} $w$, denoted by $\t_{(\delta)}(w)$, as the target of the maximal accordion in $\{\nu \mid \nu \in \opResAc'(\delta),\ s(\nu) = w\}$.
\end{definition}

We can make $\s_{(\delta)}(v)$ and $\t_{(\delta)}(v)$ explicit.

\begin{lemma}\label{rem:useful_up_s}
Let $\delta \in \Accord'$, $v \in \NP(\delta)_0^{{\gsquare}}$ and $w \in \NP(\delta)_0^{{\rsquare}}$. The following assertions hold:
\begin{enumerate}[label=$(\roman*)$, itemsep=1mm]
    \item we have $\displaystyle \s_{(\delta)}(v)= \begin{cases}
    s(\delta) &  \begin{matrix} 
    \text{if there exists } \rho \in \Prj(\Delta^{\gpoint}) \text{ such that both} \\
    t(\rho) = v \text{ and } \rho \text{ crosses } \delta; \hfill \end{matrix} \\
    w_R(\delta) & \text{otherwise;}
    \end{cases}$
    \item we have $\displaystyle \t_{(\delta)}(w)= \begin{cases}
    t(\delta) & \begin{matrix} 
    \text{if there exists } \rho \in \Prj(\Delta^{\gpoint}) \text{ such that both} \\
    s(\rho) = w \text{ and } \rho \text{ crosses } \delta; \hfill \end{matrix} \\
    w_L(\delta) & \text{otherwise.}
    \end{cases}$
\end{enumerate}
\end{lemma}

\begin{definition} \label{def:sourceproxy}
Let  $(\delta, \eta) \in (\Accord')^2$ be a pair of noncrossing accordions, such that $\NP(\delta)\cap \NP(\eta)\neq \varnothing$, and $\col_\delta$ and $\col_\eta$ match. Without loss of generality, consider $\delta$ above $\eta$. Let $v \in \NP(\eta)_0 \cap \NP(\delta)_0^{{\gsquare}}$, and $w \in \NP(\delta)_0 \cap \NP(\eta)_0^{{\rsquare}}$. The \new{upper source of} $(\delta,\eta)$ \new{associated to} $v$, denoted by $\s_{(\delta,\eta)}(v)$, is defined as $\s_{(\delta)}(v)$. We define the \new{upper target of} $(\delta,\eta)$ \new{associated to} $w$, denoted by $\t_{(\delta,\eta)}(w)$, to be $\t_{(\eta)}(w)$.
\end{definition}

We can show that $\s_{(\delta,\eta)}(v)$ and $\t_{(\delta,\eta)}(w)$ only depends on the pair $(\delta,\eta)$ (see \cite{DS252}). Therefore, we simply call \new{upper source of $(\delta,\eta)$} the one associated to any $v \in \NP(\eta)_0 \cap \NP(\delta)_0^{{\gsquare}}$, and we denote it by \new{$\s_{(\delta, \eta)}$}.  We also call $\t_{(\delta, \eta)}$ the \new{upper target for $(\delta, \eta)$} similarly.

\begin{figure}[!ht]
\centering 
    \begin{tikzpicture}[mydot/.style={
					circle,
					thick,
					fill=white,
					draw,
					outer sep=0.5pt,
					inner sep=1pt
				}, scale = 1]
		\tikzset{
		osq/.style={
        rectangle,
        thick,
        fill=white,
        append after command={
            node [
                fit=(\tikzlastnode),
                orange,
                line width=0.3mm,
                inner sep=-\pgflinewidth,
                cross out,
                draw
            ] {}}}}
		\draw[line width=0.7mm,black] (0,0) ellipse (4cm and 1.5cm);
		\foreach \X in {0,1,...,23}
		{
		\tkzDefPoint(4*cos(pi/12*\X),1.5*sin(pi/12*\X)){\X};
		};
		
		\draw[line width=0.9mm ,bend left =60,red](1) edge (3);
		\draw[line width=0.9mm ,bend right =60,red](5) edge (3);
		\draw[line width=0.9mm ,bend right =30,red](5) edge (23);
		\draw[line width=0.9mm ,bend left =20,red](21) edge (23);
		\draw[line width=0.9mm ,bend left =20,red](7) edge (9);
		\draw[line width=0.9mm ,bend left =40,red](19) edge (21);
		\draw[line width=0.9mm ,bend left =30,mypurple,densely dashdotted](9) edge (19);
		\draw[line width=0.9mm ,bend left =40,mypurple,densely dashdotted](9) edge (17);
		\draw[line width=0.9mm ,bend left =40,red](9) edge (15);
		\draw[line width=0.9mm ,bend left =40,red](9) edge (11);
		\draw[line width=0.9mm ,bend left=30,red](13) edge (15);
		
		\draw[line width=0.7mm ,bend right=10,blue, loosely dashed](5) edge (17);
		\draw[line width=0.7mm ,bend left=10,blue, loosely dashed](7) edge (13);
		
		\filldraw [fill=mypurple,opacity=0.1] (9) to [bend left=30] (19) to [bend left=10] (17) to [bend right=40] cycle ;

		\foreach \X in {1,3,...,23}
		{
		\tkzDrawPoints[fill =red,size=4,color=red](\X);
		};

		\tkzDefPoint(-2.1,0.5){gamma};
		\tkzLabelPoint[blue](gamma){\Large $\delta$}
		\tkzDefPoint(-0.6,0.1){gamma};
		\tkzLabelPoint[blue](gamma){\Large $\eta$}
		\tkzDefPoint(-4.3,-0.5){s};
		\tkzLabelPoint[red](s){\Large $\pmb{s}_{(\delta,\eta)}$}
		\tkzDefPoint(1,2.1){t};
		\tkzLabelPoint[red](t){\Large $\pmb{t}_{(\delta,\eta)}$}
    \end{tikzpicture}
\caption{\label{fig:exuppersourceandtarget} Example of the determination of the upper source and target of a pair $(\delta, \eta) \in (\Accord')^2$.}
\end{figure}
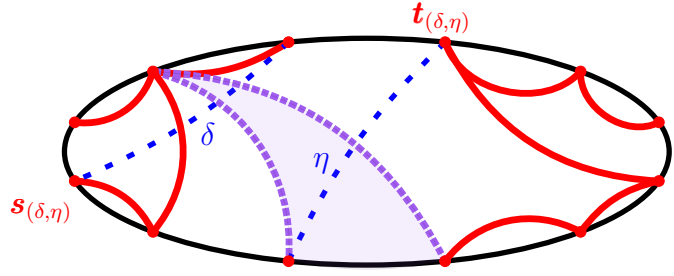

\begin{definition} \label{def:CoZcompl}
A pair $(\delta, \eta) \in (\Accord')^2$ \new{admits a $\CoZ$-completion} if the following assertions hold:
\begin{enumerate}[label=$\bullet$,itemsep=1mm]
\item $\delta$ and $\eta$ are not crossing;
\item $\NP(\delta)\cap \NP(\eta)\neq \varnothing$; 
\item there exists $\rho \in \Prj(\Delta^{\gpoint})$ such that $\rho \not \Extperp \delta$ and $\rho \not \Extperp \eta$; and,
\item $\delta$ is above $\eta$.
\end{enumerate}
In such a case, the $\rpoint$-arc $\xi$ such that $s(\xi)=\s_{(\delta,\eta)}$ and $t(\xi)=\t_{(\delta,\eta)}$ is an accordion, and we call it the \new{$\CoZ$-completion} of $(\delta,\eta)$.
\end{definition}

\begin{lemma} \label{lem:welldefinedCoZ}
Let $(\delta, \eta) \in (\Accord')^2$. Assume that $(\delta, \eta)$ admits a $\CoZ$-completion $\xi$. Then $\xi \in \Accord$.
\end{lemma}

\begin{definition} \label{def:CoZ}
For any $(\delta, \eta) \in (\Accord')^2$, we define the \new{$\CoZ$-move} of $(\delta,\eta)$ as follows:
\[\CoZ(\delta,\eta) = \begin{cases}
\{\delta, \eta, \xi\} & \text{if } (\delta,\eta) \text{ admits a }\CoZ\text{-completion } \xi; \\
\{\delta,\eta\} & \text{otherwise.}
\end{cases}\]
\end{definition}

\begin{figure}[ht!]
\centering 
    \begin{tikzpicture}[mydot/.style={
					circle,
					thick,
					fill=white,
					draw,
					outer sep=0.5pt,
					inner sep=1pt
				}, scale = 1]
		\tikzset{
		osq/.style={
        rectangle,
        thick,
        fill=white,
        append after command={
            node [
                fit=(\tikzlastnode),
                orange,
                line width=0.3mm,
                inner sep=-\pgflinewidth,
                cross out,
                draw
            ] {}}}}
		\draw[line width=0.7mm,black] (0,0) ellipse (4cm and 1.5cm);
		\foreach \X in {0,1,...,23}
		{
		\tkzDefPoint(4*cos(pi/12*\X),1.5*sin(pi/12*\X)){\X};
		};
		
		\draw[line width=0.9mm ,bend left =60,red](1) edge (3);
		\draw[line width=0.9mm ,bend right =60,red](5) edge (3);
		\draw[line width=0.9mm ,bend right =30,red](5) edge (23);
		\draw[line width=0.9mm ,bend left =20,red](21) edge (23);
		\draw[line width=0.9mm ,bend left =20,red](7) edge (9);
		\draw[line width=0.9mm ,bend left =40,red](19) edge (21);
		\draw[line width=0.9mm ,bend left =30,mypurple,densely dashdotted](9) edge (19);
		\draw[line width=0.9mm ,bend left =40,mypurple,densely dashdotted](9) edge (17);
		\draw[line width=0.9mm ,bend left =40,red](9) edge (15);
		\draw[line width=0.9mm ,bend left =40,red](9) edge (11);
		\draw[line width=0.9mm ,bend left=30,red](13) edge (15);
		
		\draw[line width=0.7mm ,bend right=10,blue, loosely dashed](5) edge (17);
		\draw[line width=0.7mm ,bend left=10,blue, loosely dashed](7) edge (13);
		\draw[line width=0.7mm ,bend left=10,orange, densely dashdotted](5) edge (13);
		
		\filldraw [fill=mypurple,opacity=0.1] (9) to [bend left=30] (19) to [bend left=10] (17) to [bend right=40] cycle ;

		\foreach \X in {1,3,...,23}
		{
		\tkzDrawPoints[fill =red,size=4,color=red](\X);
		};

		\tkzDefPoint(-2.1,0.6){gamma};
		\tkzLabelPoint[blue](gamma){\Large $\delta$}
		\tkzDefPoint(-.8,0.1){gamma};
		\tkzLabelPoint[blue](gamma){\Large $\eta$}
		\tkzDefPoint(-0.2,1.5){gamma};
		\tkzLabelPoint[orange](gamma){\Large $\xi$}
		\tkzDefPoint(-4.2,-0.5){s};
		\tkzLabelPoint[red](s){\Large $\pmb{s}_{(\delta,\eta)}$}
		\tkzDefPoint(1,2.1){t};
		\tkzLabelPoint[red](t){\Large $\pmb{t}_{(\delta,\eta)}$}
    \end{tikzpicture}
\caption{\label{fig:Co-Z} Construction of the $\CoZ$-completion $\xi \in \Accord$ of $(\delta, \eta)$.}
\end{figure}
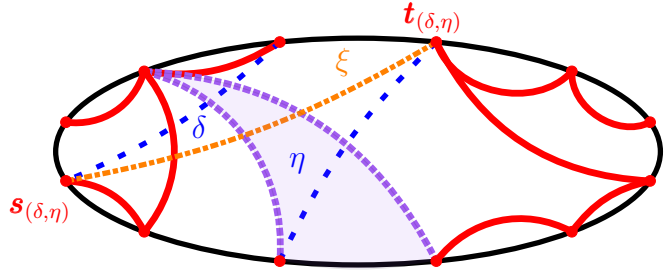

\begin{prop}
\label{prop:Co-Z}
Let $(\delta, \eta) \in (\Accord')^2$ admit a $\CoZ$-completion $\xi$. Then $\xi \in \opResAc'(\delta, \eta)$. 
\end{prop}

	\section{Combinatorial realization of the Auslander--Reiten correspondence}
	\label{sec:Tilt}
	\pagestyle{plain}

For $\Lambda$ an Artin algebra of finite global dimension, M. Auslander and I. Reiten \cite{Auslander1991} establish a one-to-one correspondence from cotilting objects (of finite projective dimension) to contravariantly finite resolving subcategories of the category of all finitely generated $\Lambda$-modules.

\begin{theorem}[{\cite[Cor. 5.6]{Auslander1991} }]
\label{thm:Auslander}
Let $\Lambda$ be an Artin algebra. Assume that $\Lambda$ is of finite global dimension. The morphism $T \mapsto {}^{\perp_>{0}}T$ gives a one-to-one correspondence between isomorphism classes of basic cotilting modules and contravariantly finite resolving subcategories of $\mod{\Lambda}$
\end{theorem}

Assuming that $(Q,R)$ is representation-finite, all the subcategories of $\rep(Q,R)$ are functorially finite. 
In our setting, we can state this result as follows.

\begin{cor} \label{cor:Auslander} Let $(Q,R)$ be a representation-finite gentle quiver. The map: \[ \left\{\begin{matrix}
\Tilt(Q,R) & \longrightarrow & \ResOrd(Q,R) \\
\mathfrak{T} & \longmapsto & {}^{\perp_{> 0}} \left(\mathfrak{T}^{\perp_{> 0}} \right)
\end{matrix} \right.\]
is bijective.
\end{cor}

\begin{proof}
Let $\mathfrak{T}$ be a tilting object of $\rep(Q,R)$. Using the dual of $\cref{thm:Auslander}$, $\mathfrak{T}$ is in bijection with a coresolving subcategory $\mathfrak{T}^{\perp_{>0}}$. Coresolving subcategories are in bijection with the resolving subcategory completing the hereditary cotorsion pair ${}^{\perp_{> 0}} \left(\mathfrak{T}^{\perp_{> 0}} \right)$.
\end{proof}

In this section, we construct a combinatorial inverse to this map using tools we developed, which allows us to describe any resolving subcategories in terms of their upper join decomposition when $(Q,R)$ is a gentle tree.

\begin{conv} \label{conv:gentletree} Fix a gentle tree $(Q,R)$, and let $(\pmb{\Sigma},\mathcal{M}, \Delta^{\gpoint})$ be its associated dualizable $\gpoint$-dissected marked disc. Given a resolving subcategory $\mathscr{R} \subseteq \rep_\mathbb{K}(Q,R)$ associated with a set of accordions $\pmb{\mathscr{R}}$, we write:
\begin{enumerate}[label = $\bullet$, itemsep=0.1em]
    \item $\mathfrak{I}_\mathscr{R}$ for the set of all the  non-projective indecomposable representations in $\mathscr{R}$: we recall that $\mathfrak{I}_\mathscr{R}$ is a $\Res$-closed ideal of $(\pmb{\ind \setminus \proj}(Q,R), \Resleq)$ by \cref{prop:folkloreres};
    \item $\bbDelta'_{(\mathscr{R})}$ for the set of accordions in $\bbDelta_{(\mathfrak{I}_\mathscr{R})} \setminus \Prj(\Delta^{\gpoint})$: $\bbDelta'_{(\mathscr{R})}$ is a $\opResAc'$-closed ideal of $(\Accord', \Accordleq)$ by \cref{prop:folkloreresaccord}. Where $\opResAc'$ is the set of nonprojective accordions associated with the indecomposable objects of $\mathscr{R}$;
    \item $\mathfrak{E}_\mathscr{R}$ for the antichain of maximal elements of $\mathfrak{I}_\mathscr{R}$ in $(\pmb{\ind \setminus \proj}(Q,R), \Resleq)$; and,
    \item $\mathfrak{D}_\mathscr{R}$ for the antichain of maximal elements of $\bbDelta'_{(\mathscr{R})}$ in $(\Accord', \Accordleq)$.
\end{enumerate}
\end{conv} 
Note that the sets appearing in the first two points are the same up to the correspondence between indecomposable objects and accordions. 

\subsection{Canonical rigid collection}
\label{ss:canonrigid} 
In this subsection, we prove that $\mathfrak{E}_{\mathscr{R}}$ can be rigidified in a canonical way, which defines another antichain $\mathfrak{E}^{\mathrm{r}}_{\mathscr{R}}$ such that $\opResAc'\left(\mathfrak{E}^{\mathrm{r}}_{\mathscr{R}} \right) = \mathfrak{I}_{\mathscr{R}}$. Then, from the indecomposable representations in $\mathfrak{E}^{\mathrm{r}}_{\mathscr{R}}$, we construct the unique tilting collection $\mathfrak{T}_\mathscr{R}$ such that $\Res(\mathfrak{T}_{\mathscr{R}})=\mathscr{R}$. This will guarantee the unicity of $\mathfrak{E}_{\mathscr{R}}^{\mathrm{r}}$.

Recall that given $X,Y \in \ind(Q,R)$, we write $X \THomleq Y$ whenever there exists a finite sequence of nonzero morphisms from $X$ to $Y$. As $\rep(Q,R)$ is \emph{directed}, the relation $\THomleq$ is an order on $\ind(Q,R)$. We refer the reader to \cite[Section 5.3]{DS251} for more details. We recall the following result.

\begin{prop}[\cite{DS251}]\label{prop:ResandHom}
    Let $X,Y \in \pmb{\ind \setminus \proj}(Q,R)$. If $\Res(X) \subseteq \Res(Y)$, then $X \THomleq Y$.
\end{prop}

\begin{lemma}
\label{lem:Ext_max}
Let $(Q,R)$ be a gentle tree. Consider a resolving subcategory $\mathscr{R} \subseteq \rep(Q,R)$. Let $(M,N) \in (\mathfrak{E}_{\mathscr{R}})^2$ and $i\in\mathbb{N}$ such that there exists a non-split short exact sequence: \[\begin{tikzcd}
	M & E & \omega^i(N),
	\arrow[tail, from=1-1, to=1-2]
	\arrow[two heads, from=1-2, to=1-3]
\end{tikzcd}\]  
with some $E = \bigoplus_{j=1}^k  E_j \in \rep(Q,R)$, $k \in \{1,2\}$ and $\omega^i(N)$ a nonprojective indecomposable summand of $\Omega^i(N)$. Then, for all $j \in \{1,\ldots,k\}$, if $E_j$ is nonprojective there exists $K_j \in \mathfrak{C}_\mathscr{R}$ such that $E_j \Resleq K_j$, and $M \THomleq K_j$.
\end{lemma}

\begin{proof}
Let $j \in \{1,\ldots, k\}$. As $\mathscr{R}$ is resolving, because $E_j$ is nonprojective, we have that $E_j \in \mathfrak{I}_\mathscr{R}$. By \cref{prop:folkloreres}, $\mathfrak{I}_\mathscr{R}$ is an ideal, and so there exists $K_j \in \mathfrak{E}_\mathscr{R}$ such that $E_j \Resleq K_j$. By \cref{prop:ResandHom}, we have that $E_j \THomleq K_j$. Since $M \THomleq E_j$, transitivity yields $M \THomleq K_j$.
\end{proof}

We introduce an algorithm that yields a rigid collection $\mathfrak{E}^{\mathrm{r}}_\mathscr{R}$ from $\mathfrak{E}_\mathscr{R}$. 
\begin{algo} \label{algo:rigidER}
Let $(Q,R)$ be a gentle tree, and $\mathscr{R} \subseteq \rep(Q,R)$ be a resolving subcategory. 
\begin{enumerate}[label=$(\arabic*)$, itemsep=1mm]
    \item We input $\mathfrak{F}_0 = \mathfrak{E}_\mathscr{R}$.
    \item At the $j$th iteration of the algorithm, we define $\mathfrak{F}_{j+1}$ from $\mathfrak{F}_j$ as follows: let $\mathfrak{X}_j$ be the set of $M \in \mathfrak{F}_j$ such that there exists $N \in \mathfrak{F}_j$ and a non-split short exact sequence: \[\begin{tikzcd}
	M & E & \omega^i(N),
	\arrow[tail, from=1-1, to=1-2]
	\arrow[two heads, from=1-2, to=1-3]
\end{tikzcd}\]
with $E \in \rep(Q,R)$ and $\omega^i(N)$ a nonprojective indecomposable summand of $\Omega^i(N)$ for some $i \in \mathbb{N}$. 
\begin{enumerate}[label=$(2 \alph*)$, itemsep=1mm]
\item If $\mathfrak{X}_j \neq \varnothing$, we let $M_j$ be a minimal element of $\mathfrak{X}_j$ with respect to $\THomleq$, and we set $\mathfrak{F}_{j+1} = \mathfrak{F}_j  \setminus \{M_j\}$; 

\item otherwise, we set $\mathfrak{F}_{j+1} = \mathfrak{F}_j$.
\end{enumerate}
    \item If $\mathfrak{F}_{j+1} \neq \mathfrak{F}_j$, then go back to Step $(2)$;
    \item otherwise, return $\mathfrak{F}_{j+1}$.
\end{enumerate}
\end{algo}

\begin{prop}
\label{prop:rigidResgen}
Let $(Q,R)$ be a gentle tree. For any $\mathscr{R} \in \ResOrd(Q,R)$, \cref{algo:rigidER} terminates, and returns a collection $\mathfrak{F}$ such that all the following assertions hold:
\begin{enumerate}[label=$(\roman*)$, itemsep=1mm]
    \item $\mathfrak{F}$ is rigid;
    \item  $\mathfrak{F}$ is an antichain of $(\pmb{\ind \setminus \proj}(Q,R), \Resleq)$; and,
    \item $\Res(\mathfrak{F})=\mathscr{R}$.
\end{enumerate}
\end{prop}
\begin{proof}
Let $\mathscr{R}$ be as assumed. Since $\mathfrak{E}_\mathscr{R}$ is finite, the algorithm terminates. Let $j \geqslant 0$ be the smallest index for which $\mathfrak{F}_j = \mathfrak{F}_{j+1} = \mathfrak{F}$. The fact that  $\mathfrak{F}$ satisfies $(i)$ and $(ii)$ is obvious by construction. Let us show by induction that $\mathfrak{F}$ satisfies $(iii)$.

By hypothesis, $\Res(\mathfrak{E}_\mathscr{R}) = \mathscr{R}$. Let us assume by induction that, for some $\ell \geqslant 0$, we have $\Res(\mathfrak{F}_\ell) = \mathscr{R}$. Assume that $\mathfrak{X}_\ell \neq \varnothing$, and let $M_\ell$ be a minimal element of $\mathfrak{X}_\ell$ with respect to $\THomleq$. We will check that $M_\ell \in \Res(\mathfrak{F_{\ell+1}})$, which will imply the result. Let us consider a nonsplit short exact sequence: \[\begin{tikzcd}
	M_\ell & E & \omega^i(N),
	\arrow[tail, from=1-1, to=1-2]
	\arrow[two heads, from=1-2, to=1-3]
\end{tikzcd}\] with $E = \oplus_{s=1}^k E_s \in \rep(Q,R)$, for $k \in \{1,2\}$, $N \in \mathfrak{F}_\ell$ and $\omega^i(N)$ an indecomposable nonprojective summand of $\Omega^i(N)$, for some $i \in \mathbb{N}$. Then, by \cref{lem:Ext_max}, if $E_k$ is nonprojective there exists $K_k \in \mathfrak{E}_\mathscr{R}$ such that $M_\ell \THomleq K_i$ and $E_k\Resleq K_k$.  We can still state that $M_\ell \THomleq N$ and $\omega^i(N)\Resleq N$. It follows that $K_k\neq M_l$ else the extension is split. The same statement holds for $N$. It follows that $K_k\in\Res(\mathfrak{F}_{\ell+1})$ and $N\in\Res(\mathfrak{F}_{\ell+1})$.

Therefore $M$ is a kernel of an epimorphism of the subcategory $\Res(\mathfrak{F}_{\ell+1})$ as $K_\ell$ and $N$ are in $\Res(\mathfrak{F}_{\ell+1})$. It follows that $\Res(\mathfrak{F}_{\ell+1}) = \Res(\mathfrak{F}_\ell)$ and, by induction hypothesis, we get that $\Res(\mathfrak{F}_j) = \mathscr{R}$. We get that $\mathfrak{F}$ satisfies $(iii)$.
\end{proof}

\begin{remark} \label{rem:anywayalgorigid} In Step $(2)$ of \cref{algo:rigidER}, we could choose to delete any element of $\mathfrak{X}_i$. Since $\THomleq$ is an order on $\ind(Q,R)$, we obtain the same antichain $\mathfrak{F}$.
\end{remark}

\begin{definition}
\label{def:rigidmax}
Let $(Q,R)$ be a gentle tree, and $\mathscr{R} \in \ResOrd(Q,R)$. We call \new{the canonical rigid collection of $\mathscr{R}$} the rigid collection $\mathfrak{E}_\mathscr{R}^{\mathrm{r}}$ obtained as a result of \cref{algo:rigidER} on $\mathscr{R}$.
\end{definition}

\begin{cor} \label{cor:rigidERAccord}
Let $(Q,R)$ be a gentle tree, and set $(\pmb{\Sigma},\mathcal{M}, \Delta^{\gpoint}) = \Surf(Q,R)$. Consider a resolving subset $\mathscr{R} \subseteq \Accord'$. Then \cref{algo:rigidER}, rewritten in terms of the geometric model, using \cref{prop:Extorthoaccord}, returns a collection $\mathfrak{D}_\mathscr{R}^{\mathrm{r}}$ such that:
\begin{enumerate}[label=$(\roman*)$, itemsep=1mm]
    \item $\mathfrak{D}_\mathscr{R}^r$ is a rigid admissible $\rpoint$-dissection of $(\pmb{\Sigma}, \mathcal{M})$;
    \item $\mathfrak{D}_\mathscr{R}^\mathrm{r} \in \Anti(\Accord', \Accordleq)$; and,
    \item $\opResAc(\mathfrak{D}_\mathscr{R}^{\mathrm{r}}) = \mathscr{R}$.
\end{enumerate}
We call $\mathfrak{D}_\mathscr{R}^{\mathrm{r}}$ the \new{canonical geometric rigid collection} of $\mathscr{R}$.
\end{cor}

\begin{proof}
This is a consequence of \cref{prop:rigidResgen}, by noticing that $\mathfrak{D}_\mathscr{R}^\mathrm{r} = \bbDelta_{(\mathfrak{E}_\mathscr{R}^{\mathrm{r}})}$, and of \cref{thm:tiltgeom}.
\end{proof}

\subsection{Projective completions}
\label{ss:projcompl}

In this subsection, we introduce a \emph{projective completion} operation, which consists of adding projective accordions that are Ext-orthogonal to all the accordions in a given subset of $\Accord'$. We focus on the projective completion of a rigid collection and highlight the behavior of the cells that prevents it from being a tilting collection, in comparison with well-chosen projective accordions. In the next subsection, we use these results to gradually transform a rigid collection into a tilting collection while preserving the resolving subcategory it generates.

\begin{definition} \label{def:projcompl}
    Let $\mathfrak{D} \subseteq \Accord'$. We define the \new{projective completion} $\widehat{\mathfrak{D}}$ of $\mathfrak{D}$ as the union of $\mathfrak{D}$ with the projective accordions $\rho$ that are $\Ext$-orthogonal to all the accordions in $\mathfrak{D}$.
\end{definition}

Given $\mathfrak{D} \subseteq \Accord'$, one can construct $\widehat{\mathfrak{D}}$ combinatorially thanks to the following result, which is a restatement of $(i)$ and $(iii)$ of \cref{prop:CombidescripNproj}.

\begin{lemma} \label{lem:projextortho}
    Let $\delta \in \Accord'$. Then $\rho \in \Prj(\Delta^{\gpoint})$ is not Ext-orthogonal to $\delta$ if and only if either:
    \begin{enumerate}[label=$\bullet$, itemsep=1mm]
        \item $\rho$ crosses $\delta$; or,
        \item $\rho$ and $\delta$ share a common endpoint $u \in \mathcal{M}_{\rpoint}$, and $\rho \prec_u \delta$.
    \end{enumerate}
\end{lemma}

The following statement is an obvious consequence of the definition of projective completions, \cref{lem:projextortho} and \cref{thm:tiltgeom}.

\begin{lemma} \label{lem:projcompladmdiss}
    Let $\mathfrak{D} \subseteq \Accord'$. The following assertions hold:
    \begin{enumerate}[label=$\bullet$, itemsep=1mm]
        \item $\mathfrak{D}$ is an admissible $\rpoint$-dissection if and only if so is $\widehat{\mathfrak{D}}$; and,
        \item $\mathfrak{D}$ is rigid if and only if so is $\widehat{\mathfrak{D}}$.
    \end{enumerate}
\end{lemma}

Now, given $\mathfrak{D}_{\mathscr{R}}^{\mathrm{r}}$, we explain categorically why $\widehat {\mathfrak{D}_{\mathscr{R}}^{\mathrm{r}}} \subset \mathfrak{T}_\mathscr{R}$. To do so, let us prove the following statement, which is available in the setting of abelian categories.

\begin{prop} \label{prop:projExtorthosummands} 
   Let $\mathscr{R} \subseteq \rep(Q,R)$ be a resolving subcategory. Then all the projective accordions that are $\Ext$-orthogonal to $\mathfrak{D}_{\mathscr{R}}^{\mathrm{r}}$ are in $\bbDelta_{(\mathfrak{T}_\mathscr{R})}$.
\end{prop}

\begin{proof}
    By \cref{prop:projorthgen}, in the geometric model, any projective accordion Ext-orthogonal $\mathfrak{D}_\mathscr{R}^{\mathrm{r}}$ is Ext-orthogonal to all the objects in $\bbDelta_\mathscr{R}' = \opResAc(\mathfrak{D}_\mathscr{R}^{\mathrm{r}})$. Such a projective accordion is Ext-orthogonal to $\mathfrak{T}_\mathscr{R}$, and, therefore, lives in $\bbDelta_{(\mathfrak{T}_\mathscr{R})}$.
\end{proof}

From $\widehat{\mathfrak{D}_{\mathscr{R}}^{\mathrm{r}}}$, we will proceed by gradually adding accordions to construct $\bbDelta_{(\mathfrak{T}_\mathscr{R})}$. Note that we do not know yet that $\mathfrak{D}_{\mathscr{R}}^{\mathrm{r}} \subseteq \bbDelta_{(\mathfrak{T}_\mathscr{R})}$.

\subsection{Cell configurations in rigid projective completions}
\label{ss:Cellrigiddissec}

Let $v,w \in \mathcal{M}_{\gpoint}$ such that $v \neq w$. An accordion $\delta \in \Accord$ \new{separates} $v$ and $w$ if $v$ and $w$ are in different connected components of $\pmb{\Sigma} \setminus \delta$. We denote by $(v|w)$ the set of accordions that separate $v$ and $w$. Set $(v | w)_{\Prj} = (v|w) \cap \Prj(\Delta^{\gpoint})$. Note that  $(v | w)_{\Prj} \neq \varnothing$, as $\Prj(\Delta^{\gpoint})$ is a dualizable $\rpoint$-dissection of the marked disc $(\pmb{\Sigma}, \mathcal{M})$.

Consider a rigid $\rpoint$-dissection $\mathfrak{D} \subset \Accord'$. By \cref{lem:projcompladmdiss}, $\widehat{\mathfrak{D}}$ is also a rigid $\rpoint$-dissection, hence admissible, of $(\pmb{\Sigma}, \mathcal{M})$ that contains $\mathfrak{D}$. Given any $C \in \pmb{\Gamma}(\widehat{\mathfrak{D}})$, recall that we write $\partial C$ for the boundary of $C$. If $\widehat{\mathfrak{D}}$ is not dualizable, then there exists $C \in \pmb{\Gamma}(\widehat{\mathfrak{D}})$ such that $ \#(\mathcal{M}_{\gpoint} \cap \partial C) >1$. We set: \[\bbGamma(\widehat{\mathfrak{D}}) = \{ C \in \pmb{\Gamma}(\widehat{\mathfrak{D}}) \mid  \#(\mathcal{M}_{\gpoint} \cap \partial C) > 1\}.\]

In the following, we establish some results on the behavior of cells in $\bbGamma(\widehat{\mathfrak{D}})$ compared to the projective accordions that separate the  $\gpoint$-points in these cells. In the next subsection, we use these results to gradually construct the desired tilting collection $\mathfrak{T}_\mathscr{R}$ from $\widehat{\mathfrak{D}_\mathscr{R}^{\mathrm{r}}}$.

\begin{lemma} \label{lem:cellboundary}
Let $\mathfrak{D} \subseteq \Accord'$ be an admissible $\rpoint$-dissection of $(\pmb{\Sigma},\mathcal{M})$. Then, for any $C \in \bbGamma(\widehat{\mathfrak{D}})$, there exists $\delta \in \mathfrak{D}$ such that $\delta \subset \partial C$.
\end{lemma}

\begin{proof}
Let $C \in \bbGamma(\widehat{\mathfrak{D}})$. Let $v,w \in \mathcal{M}_{\gpoint} \cap \partial C$ be two distinct points. As $\Prj(\Delta^{\gpoint})$ is a dualizable $\rpoint$-dissection of $(\pmb{\Sigma},\mathcal{M})$, there exists $\rho \in (v|w)_{\Prj}$, and such a $\rho$ does not intersect any projective accordions in $\partial C$. As $\rho \notin \widehat{\mathfrak{D}}$, there exists $\eta \in \mathfrak{D}$ such that $\eta \not\Extperp \rho$. By \cref{lem:projextortho}, either $\rho$ crosses $\eta$, or they share a common endpoint $u \in \mathcal{M}_{\rpoint}$, and $\rho \prec_u \eta$.

In the first case, as $(\pmb{\Sigma}, \mathcal{M})$ is a disc, this means that $\rho$ leaves $C$ by crossing an accordion $\delta \in \partial C$. So we must have $\delta \in \mathfrak{D}$.

In the second case, assume by contradiction that all the accordions in $\partial C$ are projective accordions. Therefore $\rho \subset C$. As $\eta \nsubseteq \partial C$, there exists $\rho' \in \partial C$ such that there exists $u \in \mathcal{M}_{\rpoint}$ such that $\rho,\rho',\eta \in \Accord_{(u)}$, and $\rho \prec_u \rho' \prec_u \eta$. However, by \cref{lem:projextortho}, $\rho' \not\Extperp \eta$. This contradicts the fact that $\rho' \in \widehat{\mathfrak{D}}$. 
\end{proof}

\begin{lemma} \label{lem:cellbehaviorwithaccordions}
    Let $\mathfrak{D} \subseteq \Accord'$ be an admissible $\rpoint$-dissection of $(\pmb{\Sigma}, \mathcal{M})$. Let $C \in \pmb{\Gamma}(\widehat{\mathfrak{D}})$. Consider three distinct accordions $\delta,\eta,\mu \in \Accord$ such that $\delta, \eta, \mu \subset \partial C$. Then $\eta$ and $\mu$ must be contained in the same connected component of $\pmb{\Sigma} \setminus \delta$.
\end{lemma}

\begin{proof}
    This result follows from the fact that $(\pmb{\Sigma},\mathcal{M})$ is a marked disc without puncture.
\end{proof}

\begin{lemma} \label{lem:cellbehaviorwithseparation}
    Let $\mathfrak{D} \subseteq \Accord'$ be an admissible $\rpoint$-dissection of $(\pmb{\Sigma}, \mathcal{M})$. Let $C \in \bbGamma(\widehat{\mathfrak{D}})$, and $(v,w) \in (\mathcal{M} \cap \partial C)^2$ be such that $v \neq w$. Then, for any $\rho \in (v|w)_{\Prj}$, there are either one or two accordions contained in $\partial C$ that are not $\Ext$-orthogonal to $\rho$.
\end{lemma}

\begin{proof}
    By the proof of \cref{lem:cellboundary}, there is at least one accordion satisfying this property. Assume, by contradiction, that there exist three accordions $\delta, \eta$, and $\mu$ contained in $\partial C$ that are not Ext-orthogonal to $\rho$. 
    
    If none of them crosses $\rho$, by \cref{lem:projextortho}, and by pigeonhole principle, two out of those three accordions -- let us say $\delta$ and $\eta$ --, together with $\rho$, share a common endpoint $u \in \mathcal{M}_{\rpoint}$ and they are both greater that $\rho$ for $\prec_u$. As $\prec_u$ is a total order, up to exchanging their roles, we have $\delta \prec_u \eta$. This implies a contradiction with \cref{lem:cellbehaviorwithaccordions}, since $\eta$ and $\mu$ lie in distinct connected components of $\pmb{\Sigma} \setminus \delta$.

    Assume that exactly one of those three accordions -- let us say $\delta$ -- crosses $\rho$. If $\eta$ and $\mu$ do not cross $\delta$, as $(\pmb{\Sigma}, \mathcal{M})$ is a marked disc, and by \cref{lem:projextortho}, they must share, together with $\rho$, a common endpoint. As $\eta$ and $\mu$ cannot share a common endpoint, we obtained a contradiction by \cref{lem:cellbehaviorwithaccordions} for similar reasons.

    Then at least two out of those three accordions must cross $\rho$. In such a case, as $(\pmb{\Sigma}, \mathcal{M})$ is a marked disc, there is one of those accordions -- let us say $\delta$ -- that cuts $\Sigma$ into two connected components where $\eta$ and $\mu$ are in distinct connected components, which contradicts \cref{lem:cellbehaviorwithaccordions}.
\end{proof}

So, if we want to construct a dualizable $\rpoint$-dissection of $(\pmb{\Sigma}, \mathcal{M})$ from $\widehat{\mathfrak{D}}$, by \cref{lem:projextortho,lem:cellbehaviorwithseparation}, we are reduced to treating the cases enumerated in the following corollary.

\begin{cor} \label{cor:wheretoseparatecellsallcases}
Let $\mathfrak{D} \subset \Accord'$ be an admissible $\rpoint$-dissection of $(\pmb{\Sigma}, \mathcal{M})$. Consider $C \in \bbGamma(\widehat{\mathfrak{D}})$, and $(v,w) \in (\mathcal{M}_{\gpoint} \cap \partial C)^2$ such that $v \neq w$. For any $\rho \in (v|w)_{\Prj}$, exactly one of the following cases occurs:
\begin{enumerate}[label=$(\arabic*)$, itemsep=1mm]
    \item \label{1} there is exactly one accordion $\delta \in \Accord'$ contained in $\partial C$ which is not Ext-orthogonal to $\rho$, and either:
    \begin{enumerate}[label=$(1\alph*)$, itemsep=1mm]
        \item \label{1a} $\delta$ and $\rho$ share a common endpoint $u$, and $\rho \prec_u \delta$; or,
        \item \label{1b} $\delta$ crosses $\rho$; or
    \end{enumerate}
    \item \label{2}  there are exactly two accordion $\delta, \eta \in \Accord'$ contained in $\partial C$ which are not Ext-orthogonal to $\rho$, and either:
    \begin{enumerate}[label=$(2\alph*)$, itemsep=1mm]
        \item \label{2a} by denoting $u_1$ and $u_2$ the endpoints of $\rho$, we have that $\rho$ shares $u_1$ as a common endpoint with $\delta$ with $\rho \prec_{u_1} \delta$, and $\rho$ shares $u_2$ as a common endpoint with $\eta$ with $\rho \prec_{u_2} \eta$;
        \item \label{2b} $\eta$ crosses $\rho$, and $\delta$ and $\rho$ share a common endpoint $u$ with $\rho \prec_u \delta$; or,
        \item \label{2c} $\rho$ crosses both $\delta$ and $\eta$.
    \end{enumerate}
\end{enumerate}
We illustrate all those cases in \cref{fig:wheretoseparate}.
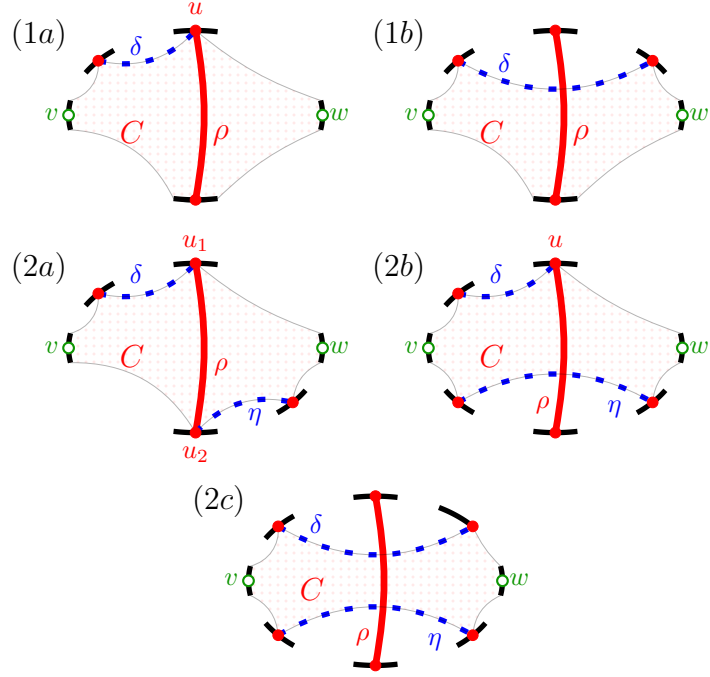
\begin{figure}[!ht]
\centering 
    \begin{tikzpicture}[mydot/.style={
					circle,
					thick,
					fill=white,
					draw,
					outer sep=0.5pt,
					inner sep=1pt
				}, scale = .8]
		\tikzset{
		osq/.style={
        rectangle,
        thick,
        fill=white,
        append after command={
            node [
                fit=(\tikzlastnode),
                orange,
                line width=0.3mm,
                inner sep=-\pgflinewidth,
                cross out,
                draw
            ] {}}}}
        \begin{scope}[scale=.7]
            		\foreach \X in {0,1,...,35}
		{
		\tkzDefPoint(3*cos(pi/18*\X),2*sin(pi/18*\X)){\X};
		};
        \draw [line width=0.7mm,domain=80:100] plot ({3*cos(\x)}, {2*sin(\x)});
        \draw [line width=0.7mm,domain=130:150] plot ({3*cos(\x)}, {2*sin(\x)});
        \draw [line width=0.7mm,domain=170:190] plot ({3*cos(\x)}, {2*sin(\x)});
        \draw [line width=0.7mm,domain=260:280] plot ({3*cos(\x)}, {2*sin(\x)});
        \draw [line width=0.7mm,domain=350:370] plot ({3*cos(\x)}, {2*sin(\x)});

		\draw[line width=0.9mm ,bend left=10,red](9) edge (27);
		
		\draw[line width=0.7mm,blue,bend right=30, loosely dashed](14) edge (9);

            \filldraw [pattern=dots, pattern color=red,opacity=0.3,line width=0mm] (9) to [bend left=30] (14) to [bend left=40] (17) to [bend right=20] (19) to [bend left=30] (26)  to [bend right=20] (28) to [bend left=10] (35) to [bend right=10] (1) to [bend left=10] cycle ;

		\foreach \X in {9,14,27}
		{
		\tkzDrawPoints[fill =red,size=4,color=red](\X);
		};

        \foreach \X in {0,18}
		{
		\tkzDrawPoints[size=4,color=dark-green,mydot](\X);
		};
		
		\tkzDefPoint(-1.4,2.1){g};
		\tkzLabelPoint[blue](g){$\delta$}
            \tkzDefPoint(.6,0){h};
		\tkzLabelPoint[red](h){\Large $\rho$}
            \tkzDefPoint(0,2.9){j};
		\tkzLabelPoint[red](j){$u$}
            \tkzDefPoint(-1.5,.1){k};
        \tkzLabelPoint[red](k){\Large $C$}
            \tkzDefPoint(-3.7,2.5){k};
		\tkzLabelPoint[black](k){\Large $(1a)$}
            \tkzDefPoint(-3.4,0.4){l};
		\tkzLabelPoint[dark-green](l){$v$}
            \tkzDefPoint(3.4,0.4){l};
		\tkzLabelPoint[dark-green](l){$w$}
        \end{scope}

    \begin{scope}[scale=.7,xshift=8.5cm]
            		\foreach \X in {0,1,...,35}
		{
		\tkzDefPoint(3*cos(pi/18*\X),2*sin(pi/18*\X)){\X};
		};
        \draw [line width=0.7mm,domain=30:50] plot ({3*cos(\x)}, {2*sin(\x)});
        \draw [line width=0.7mm,domain=80:100] plot ({3*cos(\x)}, {2*sin(\x)});
        \draw [line width=0.7mm,domain=130:150] plot ({3*cos(\x)}, {2*sin(\x)});
        \draw [line width=0.7mm,domain=170:190] plot ({3*cos(\x)}, {2*sin(\x)});
        \draw [line width=0.7mm,domain=260:280] plot ({3*cos(\x)}, {2*sin(\x)});
        \draw [line width=0.7mm,domain=350:370] plot ({3*cos(\x)}, {2*sin(\x)});

		\draw[line width=0.9mm ,bend left=10,red](9) edge (27);
		
		\draw[line width=0.7mm,blue,bend right=30, loosely dashed](14) edge (4);
        
            \filldraw [pattern=dots, pattern color=red,opacity=0.3,line width=0mm] (4) to [bend left=30] (14) to [bend left=40] (17) to [bend right=20] (19) to [bend left=30] (26)  to [bend right=20] (28) to [bend left=10] (35) to [bend right=10] (1) to [bend left=30] cycle ;

		\foreach \X in {4,9,14,27}
		{
		\tkzDrawPoints[fill =red,size=4,color=red](\X);
		};

        \foreach \X in {0,18}
		{
		\tkzDrawPoints[size=4,color=dark-green,mydot](\X);
		};
		
		\tkzDefPoint(-1.2,1.7){g};
		\tkzLabelPoint[blue](g){$\delta$}
            \tkzDefPoint(.6,0){h};
		\tkzLabelPoint[red](h){\Large $\rho$}
            \tkzDefPoint(-1.5,.1){k};
        \tkzLabelPoint[red](k){\Large $C$}
            \tkzDefPoint(-3.7,2.5){k};
		\tkzLabelPoint[black](k){\Large $(1b)$}
            \tkzDefPoint(-3.4,0.4){l};
		\tkzLabelPoint[dark-green](l){$v$}
            \tkzDefPoint(3.4,0.4){l};
		\tkzLabelPoint[dark-green](l){$w$}
        \end{scope}

         \begin{scope}[scale=.7, xshift=0cm, yshift=-5.5cm]
            		\foreach \X in {0,1,...,35}
		{
		\tkzDefPoint(3*cos(pi/18*\X),2*sin(pi/18*\X)){\X};
		};
        \draw [line width=0.7mm,domain=80:100] plot ({3*cos(\x)}, {2*sin(\x)});
        \draw [line width=0.7mm,domain=130:150] plot ({3*cos(\x)}, {2*sin(\x)});
        \draw [line width=0.7mm,domain=170:190] plot ({3*cos(\x)}, {2*sin(\x)});
        \draw [line width=0.7mm,domain=260:280] plot ({3*cos(\x)}, {2*sin(\x)});
        \draw [line width=0.7mm,domain=310:330] plot ({3*cos(\x)}, {2*sin(\x)});
        \draw [line width=0.7mm,domain=350:370] plot ({3*cos(\x)}, {2*sin(\x)});

		\draw[line width=0.9mm ,bend left=10,red](9) edge (27);
		
		\draw[line width=0.7mm,blue,bend right=30, loosely dashed](14) edge (9);
        \draw[line width=0.7mm,blue,bend left=30, loosely dashed](27) edge (32);
        
            \filldraw [pattern=dots, pattern color=red,opacity=0.3,line width=0mm] (9) to [bend left=30] (14) to [bend left=40] (17) to [bend right=20] (19) to [bend left=30] (27) to [bend left=30] (32) to [bend left=30] (35) to [bend right=10] (1) to [bend left=10] cycle ;

		\foreach \X in {9,14,27,32}
		{
		\tkzDrawPoints[fill =red,size=4,color=red](\X);
		};

        \foreach \X in {0,18}
		{
		\tkzDrawPoints[size=4,color=dark-green,mydot](\X);
		};
		
		\tkzDefPoint(-1.4,2.2){g};
		\tkzLabelPoint[blue](g){$\delta$}
            \tkzDefPoint(.6,0){h};
		\tkzLabelPoint[red](h){$\rho$}
            \tkzDefPoint(0,2.9){j};
		\tkzLabelPoint[red](j){$u_1$}
            \tkzDefPoint(-1.5,.3){k};
        \tkzLabelPoint[red](k){\Large $C$}
            \tkzDefPoint(-3.7,2.5){k};
		\tkzLabelPoint[black](k){\Large $(2a)$}
            \tkzDefPoint(-3.4,0.4){l};
		\tkzLabelPoint[dark-green](l){$v$}
            \tkzDefPoint(3.4,0.4){l};
		\tkzLabelPoint[dark-green](l){$w$}
            \tkzDefPoint(1.4,-1.2){g};
		\tkzLabelPoint[blue](g){$\eta$}
            \tkzDefPoint(0,-2.1){j};
		\tkzLabelPoint[red](j){$u_2$}
        \end{scope}

         \begin{scope}[scale=.7, xshift=8.5cm, yshift=-5.5cm]
            		\foreach \X in {0,1,...,35}
		{
		\tkzDefPoint(3*cos(pi/18*\X),2*sin(pi/18*\X)){\X};
		};
        \draw [line width=0.7mm,domain=80:100] plot ({3*cos(\x)}, {2*sin(\x)});
        \draw [line width=0.7mm,domain=130:150] plot ({3*cos(\x)}, {2*sin(\x)});
        \draw [line width=0.7mm,domain=170:190] plot ({3*cos(\x)}, {2*sin(\x)});
        \draw [line width=0.7mm,domain=210:230] plot ({3*cos(\x)}, {2*sin(\x)});
        \draw [line width=0.7mm,domain=260:280] plot ({3*cos(\x)}, {2*sin(\x)});
        \draw [line width=0.7mm,domain=310:330] plot ({3*cos(\x)}, {2*sin(\x)});
        \draw [line width=0.7mm,domain=350:370] plot ({3*cos(\x)}, {2*sin(\x)});

		\draw[line width=0.9mm ,bend left=10,red](9) edge (27);
		
		\draw[line width=0.7mm,blue,bend right=30, loosely dashed](14) edge (9);
        \draw[line width=0.7mm,blue,bend left=30, loosely dashed](22) edge (32);
        
        \filldraw [pattern=dots, pattern color=red,opacity=0.3,line width=0mm] (9) to [bend left=30] (14) to [bend left=40] (17) to [bend right=20] (19) to [bend left=30] (22) to [bend left=30] (32) to [bend left=30] (35) to [bend right=10] (1) to [bend left=10] cycle ;

		\foreach \X in {9,14,22,27,32}
		{
		\tkzDrawPoints[fill =red,size=4,color=red](\X);
		};

        \foreach \X in {0,18}
		{
		\tkzDrawPoints[size=4,color=dark-green,mydot](\X);
		};
		
		\tkzDefPoint(-1.4,2.2){g};
		\tkzLabelPoint[blue](g){$\delta$}
            \tkzDefPoint(-.3,-.9){h};
		\tkzLabelPoint[red](h){$\rho$}
            \tkzDefPoint(0,2.9){j};
		\tkzLabelPoint[red](j){$u$}
            \tkzDefPoint(-1.5,.3){k};
        \tkzLabelPoint[red](k){\Large $C$}
            \tkzDefPoint(-3.7,2.5){k};
		\tkzLabelPoint[black](k){\Large $(2b)$}
            \tkzDefPoint(-3.4,0.4){l};
		\tkzLabelPoint[dark-green](l){$v$}
            \tkzDefPoint(3.4,0.4){l};
		\tkzLabelPoint[dark-green](l){$w$}
            \tkzDefPoint(1.4,-1){g};
		\tkzLabelPoint[blue](g){$\eta$}
        \end{scope}

        \begin{scope}[scale=.7, xshift=4.25cm, yshift=-11cm]
            		\foreach \X in {0,1,...,35}
		{
		\tkzDefPoint(3*cos(pi/18*\X),2*sin(pi/18*\X)){\X};
		};
        \draw [line width=0.7mm,domain=40:60] plot ({3*cos(\x)}, {2*sin(\x)});
        \draw [line width=0.7mm,domain=80:100] plot ({3*cos(\x)}, {2*sin(\x)});
        \draw [line width=0.7mm,domain=130:150] plot ({3*cos(\x)}, {2*sin(\x)});
        \draw [line width=0.7mm,domain=170:190] plot ({3*cos(\x)}, {2*sin(\x)});
        \draw [line width=0.7mm,domain=210:230] plot ({3*cos(\x)}, {2*sin(\x)});
        \draw [line width=0.7mm,domain=260:280] plot ({3*cos(\x)}, {2*sin(\x)});
        \draw [line width=0.7mm,domain=310:330] plot ({3*cos(\x)}, {2*sin(\x)});
        \draw [line width=0.7mm,domain=350:370] plot ({3*cos(\x)}, {2*sin(\x)});

		\draw[line width=0.9mm ,bend left=10,red](9) edge (27);
		
		\draw[line width=0.7mm,blue,bend right=30, loosely dashed](14) edge (4);
        \draw[line width=0.7mm,blue,bend left=30, loosely dashed](22) edge (32);
        
        \filldraw [pattern=dots, pattern color=red,opacity=0.3,line width=0mm] (4) to [bend left=30] (14) to [bend left=40] (17) to [bend right=20] (19) to [bend left=30] (22) to [bend left=30] (32) to [bend left=30] (35) to [bend right=10] (1) to [bend left=10] cycle ;

		\foreach \X in {4,9,14,22,27,32}
		{
		\tkzDrawPoints[fill =red,size=4,color=red](\X);
		};

        \foreach \X in {0,18}
		{
		\tkzDrawPoints[size=4,color=dark-green,mydot](\X);
		};
		
		\tkzDefPoint(-1.4,1.8){g};
		\tkzLabelPoint[blue](g){$\delta$}
            \tkzDefPoint(-.3,-.9){h};
		\tkzLabelPoint[red](h){$\rho$}
            \tkzDefPoint(-1.5,.3){k};
        \tkzLabelPoint[red](k){\Large $C$}
            \tkzDefPoint(-3.7,2.5){k};
		\tkzLabelPoint[black](k){\Large $(2c)$}
            \tkzDefPoint(-3.4,0.4){l};
		\tkzLabelPoint[dark-green](l){$v$}
            \tkzDefPoint(3.4,0.4){l};
		\tkzLabelPoint[dark-green](l){$w$}
            \tkzDefPoint(1.4,-1){g};
		\tkzLabelPoint[blue](g){$\eta$}
        \end{scope}
  
    \end{tikzpicture}
\caption{\label{fig:wheretoseparate} All possible distinct behavior we can have between $\rho \in (v|w)_{\Prj}$ and the accordions non Ext-orthogonal accordions to $\rho$ contained in $\partial C$, for $C \in \protect\bbGamma(\widehat{\mathfrak{D}})$.}
\end{figure}
\end{cor}

\subsection{Induction steps}
\label{ss:inductionstep}

In the following, we present results that allow the addition of accordions to a rigid $\rpoint$-dissection $\mathfrak{D}$, preserving rigidity and the resolving subset generated by $\mathfrak{D}$.

\begin{lemma}
\label{lem:addcurve1a}
    Let $\mathfrak{D} \subseteq \Accord'$ be a rigid admissible $\rpoint$-dissection. Consider $C \in \bbGamma(\widehat{\mathfrak{D}})$. Assume that there exist $(v,w) \in (\mathcal{M}_{\gpoint} \cap \partial C)^2$ such that $v \neq w$, and $\rho \in (v|w)_{\Prj}$ satisfying $(1a)$. Let $\delta \in \Accord'$ be contained in $\partial C$, sharing with $\rho$ a common endpoint $u \in \mathcal{M}_{\rpoint}$, and satisfying $\rho \prec_u \delta$. Denote by $e \neq u$ the other endpoint of $\rho$ in $\partial C$. Fix an orientation of $\pmb{\Sigma}$ around $\delta$ and so a coloration $\col_\delta$ of $\NP(\delta)_0$ such that $e\in\NP(\delta)_0^{{\gsquare}}$. Consider $\nu$ the $\rpoint$-arc defined by $s(\nu) = \s_{(\delta)}(e)$ and $t(\nu) = e$. Then the following assertions hold:
    \begin{enumerate}[label=$(\roman*)$, itemsep=1mm]
        \item $\nu \in \opResAc'(\mathfrak{D})$; and,
        \item $\{\nu\} \cup \widehat{\mathfrak{D}}$ is an admissible $\rpoint$-dissection.
    \end{enumerate}
\end{lemma}

\begin{proof}
Let $\delta$ and $\rho$ be as above. One can either compute an arrow extension from $\delta$ to $\rho$ or from $\delta$ to the non-projective summand of the highest nontrivial syzygy of $\MM(\delta)$ that ends in $u$. Then the $\rpoint$-arc defined by $s(\nu) = \s_{(\delta)}(e)$ and $t(\nu) = e$ is the extension mentioned before. Thus $\nu \in \opResAc'(\mathfrak{D})$. Then $\{\nu\} \cup \widehat{\mathfrak{D}}$ is an admissible $\rpoint$-dissection as $\nu$ does not cross any arc of  $\widehat{\mathfrak{D}}$ and does not close an inner cell as there is a part of a border component on either side of $\nu$.
\end{proof}    

\begin{lemma}
\label{lem:addcurve1b}
    Let $\mathfrak{D} \subseteq \Accord'$ be a rigid admissible $\rpoint$-dissection. Consider $C \in \bbGamma(\widehat{\mathfrak{D}})$. Assume that  there exist $(v,w) \in (\mathcal{M}_{\gpoint} \cap \partial C)^2$ such that $v \neq w$, and $\rho \in (v|w)_{\Prj}$  satisfying $(1b)$. Let $\delta \in \Accord'$ contained in $\partial C$ crossed by $\rho$. Denote by $e$ the endpoint of $\rho$ in $\partial C$. Fix an orientation of $\pmb{\Sigma}$ around $\delta$ and so a coloration $\col_\delta$ of $\NP(\delta)_0$ such that $e\in\NP(\delta)_0^{{\gsquare}}$. Consider the $\rpoint$-arc $\nu$ defined by $s(\nu) = s(\delta)$ and $t(\nu) = e$. Then the following assertions hold:
    \begin{enumerate}[label=$(\roman*)$, itemsep=1mm]
        \item $\nu$ separates $v$ and $w$;
        \item $\nu \in \opResAc'(\mathfrak{D})$; and,
        \item $\{\nu\} \cup \widehat{\mathfrak{D}}$ is an admissible $\rpoint$-dissection.
    \end{enumerate}
\end{lemma}

\begin{proof}
    In this configuration, it is clear that $(i)$ is satisfied by construction. This also implies that $\nu \notin \widehat{\mathfrak{D}}$.

     \begin{figure}[ht!]
\centering 
    \begin{tikzpicture}[mydot/.style={
					circle,
					thick,
					fill=white,
					draw,
					outer sep=0.5pt,
					inner sep=1pt
				}, scale = 1]
		\tikzset{
		osq/.style={
        rectangle,
        thick,
        fill=white,
        append after command={
            node [
                fit=(\tikzlastnode),
                orange,
                line width=0.3mm,
                inner sep=-\pgflinewidth,
                cross out,
                draw
            ] {}}}}
        \draw [line width=0.7mm,domain=10:20] plot ({4*cos(\x)}, {1.5*sin(\x)});
		\draw [line width=0.7mm,domain=100:110] plot ({4*cos(\x)}, {1.5*sin(\x)});
		\draw [line width=0.7mm,domain=130:140] plot ({4*cos(\x)}, {1.5*sin(\x)});
		\draw [line width=0.7mm,domain=160:200] plot ({4*cos(\x)}, {1.5*sin(\x)});
		\draw [line width=0.7mm,domain=250:290] plot ({4*cos(\x)}, {1.5*sin(\x)});
		\draw [line width=0.7mm,domain=312:350] plot ({4*cos(\x)}, {1.5*sin(\x)});
		\foreach \X in {0,1,...,23}
		{
		\tkzDefPoint(4*cos(pi/12*\X),1.5*sin(pi/12*\X)){\X};
		};
		
		\draw[line width=0.9mm ,bend left =30,red] (19) edge (7);
		\draw[line width=0.9mm ,bend right =20,mypurple,densely dashdotted] (9) edge (19);
		
		\draw[line width=0.7mm ,bend right=10,blue, loosely dashed](9) edge (1);
		
		\filldraw [fill=blue,opacity=0.1] (9) to [bend left=50] (11) to [bend right=30] (13) to [bend left=30] (17) to [bend right=10] (19) to [bend left=30] (21) to [bend right=20] (23) to [bend left=30] (1) to [bend left=10] cycle ;

		\foreach \X in {1,7,9,11,13,17,19,21,23}
		{
		\tkzDrawPoints[fill =red,size=4,color=red](\X);
		};
		
		\foreach \X in {12,18,22}
		{
		\tkzDrawPoints[mydot,size=6,color=dark-green,thick,fill=white](\X);
		};

		\tkzDefPoint(-2,1.35){gamma};
		\tkzLabelPoint[blue](gamma){\Large $\delta$}
		\tkzDefPoint(2.2,0){gamma};
		\tkzLabelPoint[blue](gamma){\huge $C$}
		\tkzDefPoint(-1.6,0.3){gamma};
		\tkzLabelPoint[mypurple](gamma){\Large $\nu$}
		\tkzDefPoint(-0.6,1.3){gamma};
		\tkzLabelPoint[red](gamma){\Large $\rho$}
		\tkzDefPoint(-4.5,0.25){s};
		\tkzLabelPoint[dark-green](s){\Large $v$}
		\tkzDefPoint(3.7,-.8){t};
		\tkzLabelPoint[dark-green](t){\Large $w$}
		\tkzDefPoint(1.2,-1.5){u};
		\tkzLabelPoint[red](u){\Large $e$}
    \end{tikzpicture}
\caption{\label{fig:case1tilt} Construction of $\nu$ (the densely-dotted purple line) in the case where $\rho \in (v|w)_{\Prj}$ (the continuous red line) crosses a unique $\delta \in \mathfrak{D}$ (the  loosely-dotted blue line) contained in $\partial C$ (with $C$ the blue colored area).}
\end{figure}
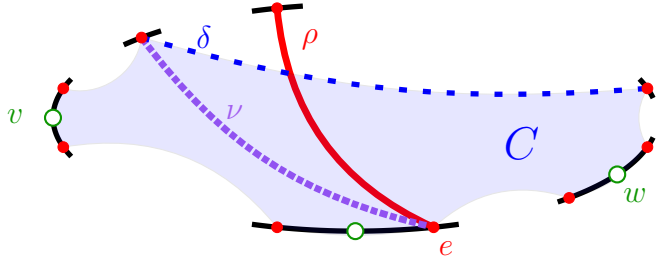

    By assumption, $\OvExt(\delta,\rho) \neq \varnothing$, and by construction, $\nu \in \OvExt(\delta,\rho)$. Therefore $\nu \in \opResAc(\delta) \subset \opResAc(\mathfrak{D})$. 
    
    Assume that $\nu \in \Prj(\Delta^{\gpoint})$. By construction, $\nu \subset C$, and therefore $\nu$ does not cross any arc within $\mathfrak{D}$. Moreover, by \cref{prop:Extorthoaccord}, we have $\nu \Extperp \delta$, and so $\wt_{s(\delta)}(\nu, \delta)=0$. Moreover, as $\mathfrak{D}$ is rigid, we have:
    \begin{enumerate}[label=$\bullet$,itemsep=1mm]
        \item for any $\delta' \in \mathfrak{D}_{[s(\delta)]}$ such that $\delta' \prec_{s(\delta)} \delta \prec_{s(\delta)} \nu$, \[\wt_{s(\delta)}(\delta',\nu) = \wt_{s(\delta)}(\delta',\delta) + \wt_{s(\delta)}(\delta,\nu) = 0;\]
        \item for any $\delta' \in \mathfrak{D}_{[s(\delta)]}$ such that $\delta \prec_{s(\delta)} \nu \prec_{s(\delta)} \delta'$, \[\wt_{s(\delta)}(\delta',\nu) = \wt_{s(\delta)}(\delta',\delta) -  \wt_{s(\delta)}(\delta,\nu) = 0;\]
        \item  moreover, for $\delta' \in \mathfrak{D}_{[e]}$, we must have $\nu \prec_e \rho \prec_e  \delta'$ as $\rho$ was satisfying $(1b)$, and so $\delta' \Extperp \nu$, and so $\wt_e(\delta',\nu) = 0$.
    \end{enumerate}
    Therefore, by \cref{prop:Extorthoaccord}, we have $\nu \Extperp \mathfrak{D}$, which leads us to a contradiction as it implies that $\nu \in \widehat{\mathfrak{D}}$.
\end{proof}
\begin{lemma} \label{lem:CoZcaseseparation1}
    Let $\mathfrak{D} \subseteq \Accord'$ be an admissible $\rpoint$-dissection of $(\pmb{\Sigma}, \mathcal{M})$. Let $C \in \bbGamma(\widehat{\mathfrak{D}})$, and $(v,w) \in (\mathcal{M} \cap \partial C)^2$ be such that $v \neq w$. Consider   $\rho \in (v|w)_{\Prj}$. If there are two accordions $\delta, \eta \in \Accord'$ contained in $\partial C$ that are not $\Ext$-ortogonal to $\rho$, then either $(\delta,\eta)$ or $(\eta,\delta)$ admits a $\CoZ$-completion.
\end{lemma}

\begin{proof}
    In such a case, the following assertions hold: 
\begin{enumerate}[label=$\bullet$, itemsep=1mm]
    \item $\delta$ and $\eta$ are not crossing by assumption on $\mathfrak{D}$;
    \item $\rho \in \NP(\delta) \cap \NP(\eta)$ as $\rho$ is having extensions with both $\delta$ and $\eta$; and, 
    \item without loss of generality, we can assume that $\delta$ is above $\eta$ with respect to the orientation we endowed $\NP(\delta) \cap \NP(\eta)$ with.
\end{enumerate}
Thus $(\delta,\eta)$ admits a $\CoZ$-completion.
\end{proof}

\begin{lemma}
\label{lem:addcurve2}
    Let $\mathfrak{D} \subseteq \Accord'$ be a rigid admissible $\rpoint$-dissection. Consider $C \in \bbGamma(\widehat{\mathfrak{D}})$. Assume there exists $(v,w) \in (\mathcal{M}_{\gpoint} \cap \partial C)^2$ such that $v \neq w$, and $\rho \in (v|w)_{\Prj}$  satisfying $(2a)$,$(2b)$ and $(2c)$. 
    Consider $\nu$ to be the $\CoZ$-completion of the pair $(\delta, \eta)$.  Then the following assertions hold:
    \begin{enumerate}[label=$(\roman*)$, itemsep=1mm]
        \item $\nu \in \opResAc'(\mathfrak{D})$; and,
        \item $\{\nu\} \cup \widehat{\mathfrak{D}}$ is an  admissible $\rpoint$-dissection.
    \end{enumerate}
\end{lemma}

 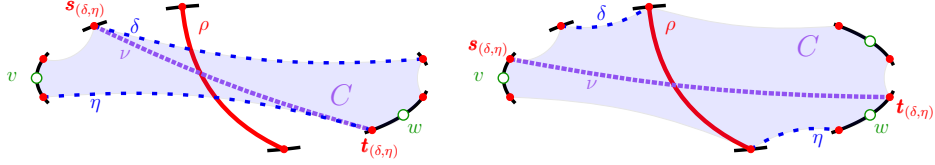
\begin{figure}[ht!]
\centering 
\scalebox{.65}{
    \begin{tikzpicture}[mydot/.style={
					circle,
					thick,
					fill=white,
					draw,
					outer sep=0.5pt,
					inner sep=1pt
				}, scale = 1]
		\tikzset{
		osq/.style={
        rectangle,
        thick,
        fill=white,
        append after command={
            node [
                fit=(\tikzlastnode),
                orange,
                line width=0.3mm,
                inner sep=-\pgflinewidth,
                cross out,
                draw
            ] {}}}}
        \draw [line width=0.7mm,domain=10:20] plot ({4*cos(\x)}, {1.5*sin(\x)});
		\draw [line width=0.7mm,domain=100:110] plot ({4*cos(\x)}, {1.5*sin(\x)});
		\draw [line width=0.7mm,domain=130:140] plot ({4*cos(\x)}, {1.5*sin(\x)});
		\draw [line width=0.7mm,domain=160:200] plot ({4*cos(\x)}, {1.5*sin(\x)});
		\draw [line width=0.7mm,domain=280:290] plot ({4*cos(\x)}, {1.5*sin(\x)});
		\draw [line width=0.7mm,domain=312:350] plot ({4*cos(\x)}, {1.5*sin(\x)});
		\foreach \X in {0,1,...,23}
		{
		\tkzDefPoint(4*cos(pi/12*\X),1.5*sin(pi/12*\X)){\X};
		};
		
		\draw[line width=0.9mm ,bend left =30,red] (19) edge (7);
		\draw[line width=0.9mm ,bend right =5,mypurple,densely dashdotted] (9) edge (21);
		
		\draw[line width=0.7mm ,bend right=10,blue, loosely dashed](9) edge (1);
		\draw[line width=0.7mm ,bend left=10,blue, loosely dashed](13) edge (21);
		
		\filldraw [fill=blue,opacity=0.1] (9) to [bend left=50] (11) to [bend right=30] (13) to [bend left=10] (21) to [bend right=20] (23) to [bend left=30] (1) to [bend left=10] cycle ;

		\foreach \X in {1,7,9,11,13,19,21,23}
		{
		\tkzDrawPoints[fill =red,size=4,color=red](\X);
		};
		
		\foreach \X in {12,22}
		{
		\tkzDrawPoints[mydot,size=6,color=dark-green,thick,fill=white](\X);
		};

		\tkzDefPoint(-2,1.35){gamma};
		\tkzLabelPoint[blue](gamma){\Large $\delta$}
		\tkzDefPoint(-2.8,-.3){gamma};
		\tkzLabelPoint[blue](gamma){\Large $\eta$}
		\tkzDefPoint(2.2,0){gamma};
		\tkzLabelPoint[mypurple](gamma){\huge $C$}
		\tkzDefPoint(-2.2,0.7){gamma};
		\tkzLabelPoint[mypurple](gamma){\Large $\nu$}
		\tkzDefPoint(-0.6,1.3){gamma};
		\tkzLabelPoint[red](gamma){\Large $\rho$}
        \tkzDefPoint(-3,1.7){gamma};
		\tkzLabelPoint[red](gamma){\Large $\s_{(\delta,\eta)}$}
        \tkzDefPoint(3,-1.1){gamma};
		\tkzLabelPoint[red](gamma){\Large $\t_{(\delta,\eta)}$}
		\tkzDefPoint(-4.5,0.25){s};
		\tkzLabelPoint[dark-green](s){\Large $v$}
		\tkzDefPoint(3.7,-.8){t};
		\tkzLabelPoint[dark-green](t){\Large $w$}
    \begin{scope}[xshift = 9.5cm]
        \draw [line width=0.7mm,domain=10:47] plot ({4*cos(\x)}, {1.5*sin(\x)});
		\draw [line width=0.7mm,domain=100:110] plot ({4*cos(\x)}, {1.5*sin(\x)});
		\draw [line width=0.7mm,domain=130:140] plot ({4*cos(\x)}, {1.5*sin(\x)});
		\draw [line width=0.7mm,domain=160:200] plot ({4*cos(\x)}, {1.5*sin(\x)});
		\draw [line width=0.7mm,domain=280:290] plot ({4*cos(\x)}, {1.5*sin(\x)});
		\draw [line width=0.7mm,domain=312:350] plot ({4*cos(\x)}, {1.5*sin(\x)});
		\foreach \X in {0,1,...,23}
		{
		\tkzDefPoint(4*cos(pi/12*\X),1.5*sin(pi/12*\X)){\X};
		};
		
		\draw[line width=0.9mm ,bend left =30,red] (19) edge (7);
		\draw[line width=0.9mm ,bend right =5,mypurple,densely dashdotted] (11) edge (23);
		
		\draw[line width=0.7mm ,bend right=20,blue, loosely dashed](9) edge (7);
		\draw[line width=0.7mm ,bend left=20,blue, loosely dashed](19) edge (21);
		
		\filldraw [fill=blue,opacity=0.1] (9) to [bend left=50] (11) to [bend right=30] (13) to [bend left=10] (19) to [bend left=20] (21) to [bend right=20] (23) to [bend left=30] (1) to [bend right=20] (3) to [bend left=10] (7) to [bend left=20] cycle ;

		\foreach \X in {1,3,7,9,11,13,19,21,23}
		{
		\tkzDrawPoints[fill =red,size=4,color=red](\X);
		};
		
		\foreach \X in {2,12,22}
		{
		\tkzDrawPoints[mydot,size=6,color=dark-green,thick,fill=white](\X);
		};

		\tkzDefPoint(-2,1.6){gamma};
		\tkzLabelPoint[blue](gamma){\Large $\delta$}
		\tkzDefPoint(2.4,-1){gamma};
		\tkzLabelPoint[blue](gamma){\Large $\eta$}
		\tkzDefPoint(2.2,1){gamma};
		\tkzLabelPoint[mypurple](gamma){\huge $C$}
		\tkzDefPoint(-2.2,0.1){gamma};
		\tkzLabelPoint[mypurple](gamma){\Large $\nu$}
		\tkzDefPoint(-0.6,1.3){gamma};
		\tkzLabelPoint[red](gamma){\Large $\rho$}
        \tkzDefPoint(-4.3,.9){gamma};
		\tkzLabelPoint[red](gamma){\Large $\s_{(\delta,\eta)}$}
        \tkzDefPoint(4.4,-.3){gamma};
		\tkzLabelPoint[red](gamma){\Large $\t_{(\delta,\eta)}$}
		\tkzDefPoint(-4.5,0.25){s};
		\tkzLabelPoint[dark-green](s){\Large $v$}
		\tkzDefPoint(3.7,-.8){t};
		\tkzLabelPoint[dark-green](t){\Large $w$}
    \end{scope}
    \end{tikzpicture}}
\caption{\label{fig:case2tilt} Two pictures illustrating the construction of $\nu$ (the densely-dotted purple line) in the case where $\rho \in (v|w)_{\Prj}$ (the continuous red line) satisfies condition $(2c)$ (on the left) and condition $(2a)$ (on the right): in both cases, $\delta,\eta \in \mathfrak{D}$ (the loosely-dotted blue line) are the two accordions contained in $\partial C$ (with $C$ the blue colored area) such that $\delta \not\Extperp \rho$ and $\eta \not\Extperp \rho$. Note that, in the $(2a)$ configuration, $\nu \notin (v|w)$.}
\end{figure}

\begin{proof}
    The fact that $\nu$ belongs to $ \opResAc'(\mathfrak{D})$ comes from \cref{prop:Co-Z}. This $\rpoint$-arc $\nu$ neither crosses any curve in $\widehat{\mathfrak{D}}$ nor closes an inner cell. 
\end{proof}

\begin{prop}
\label{prop:weightadded}
    Let $\mathfrak{D} \subseteq \Accord'$ be a rigid admissible $\rpoint$-diisection. Consider $C \in \bbGamma(\widehat{\mathfrak{D}})$. Assume that there exists $(v,w) \in (\mathcal{M}_{\gpoint} \cap \partial C)^2$ such that $v \neq w$, and $\rho \in (v|w)_{\Prj}$. In the setups $(1a)$,$(1b)$,$(2a)$,$(2b)$ and $(2c)$ let $\nu$ be the added curve defined in \cref{lem:addcurve1a,lem:addcurve1b,lem:addcurve2}. Then, for any endpoint $e$ of $\nu$, and for any $\delta \in \mathfrak{D}_{[e]}$, we have 
    $\wt_e(\delta,\nu)=0$.
\end{prop}
\begin{proof}
    Let $\delta\in\mathfrak{D}_{[e]}$.
    Consider that in each of the setups the local configuration of $\nu$ endpoints is either the same as a curve $\mu$ which enables the creation of $\nu$ and hence $\wt_e(\nu,\delta) = \wt_e(\mu,\delta)$ or $e$ is $w_R(\mu)$ or $w_L(\mu)$.
    \begin{enumerate}[label=$\bullet$, itemsep=1mm]
        \item If $e$ is an endpoint of $\mu$, then $\wt_e(\mu,\delta)=0$ as $\mathfrak{D}$ is rigid.
        \item  If $e = w_L(\mu)$ or $e = w_R(\mu)$, then $\nu$ is not crossing the two projective accordions from $\NP(\mu)$ that admit $e$ as an endpoint denoted by $\rho_1$ and $\rho_2$; by \cref{lem:counterclockandwt}, and by using \cref{thm:tiltgeom} on $\{\rho_1,\rho_2\}$, we have $\wt_e(\nu,\rho_1) \leqslant \wt_e(\rho_1,\rho_2) = 0$.
    \end{enumerate}
    In either case, we got the desired result.
\end{proof}
\begin{cor}
    \label{cor:augdissecrigid}
    Let $\mathfrak{D} \subseteq \Accord'$ be a rigid admissible $\rpoint$-dissection. Consider $C \in \bbGamma(\widehat{\mathfrak{D}})$. Assume that there exists $(v,w) \in (\mathcal{M}_{\gpoint} \cap \partial C)^2$ such that $v \neq w$, and $\rho \in (v|w)_{\Prj}$ in the setups $(1a)$,$(1b)$,$(2a)$,$(2b)$ and $(2c)$. Consider $\nu$ the added curve from \cref{lem:addcurve1a,lem:addcurve1b,lem:addcurve2}.
    Then:
    \begin{enumerate}[label=$(\roman*)$, itemsep=1mm]
        \item we have $\opResAc(\mathfrak{D} \cup \{\nu\}) = \opResAc(\mathfrak{D})$;
        \item the collection $\widehat{\mathfrak{D}}\cup \{\nu \}$ is a rigid $\rpoint$-dissection; and,
        \item there exists $(v_0, w_0) \in (\mathcal{M}_{\gpoint} \cap \partial C)^2$ such that $v_0 \neq w_0$ and $\nu \in (v_0 | w_0)$.
    \end{enumerate} 
\end{cor}
\begin{proof}
    The assertion $(i)$ occurs from \cref{lem:addcurve1a,lem:addcurve1b,lem:addcurve2}. The assertion  $(ii)$ follows from, in addition to the previous lemmas \cref{prop:weightadded}. The assertion $(iii)$ is an obvious consequence of $(ii)$ and the fact that, by construction, $\nu \subset C$: the accordion $\nu$ must, therefore, cut $C$ into two connected components that contain, in their respective boundary, at least one green-marked point.
\end{proof}

\subsection{Auslander--Reiten correspondence}
\label{ss:CombAuslander}

Now, let us consider another algorithm applied in the geometric model that allows us to realize the Auslander-Reiten bijection. 

\begin{algo} \label{algo:tiltgeom} Let $(\pmb{\Sigma},\mathcal{M}, \Delta^{\gpoint})$ be the $\gpoint$-dissected marked disc associated to a gentle tree $(Q,R)$. We set $n = \#Q_0$. Let $\mathscr{R}$ be a resolving subcategory of $(Q,R)$, and consider $\mathfrak{D}_\mathscr{R}^{\mathrm{r}}$ the canonical geometric rigid antichain of $\mathscr{R}$.

\begin{enumerate}[label=$(\arabic*)$,itemsep=1mm]
    \item We input $\mathfrak{G}_\mathscr{R}^{(0)} =  \widehat{\mathfrak{D}_\mathscr{R}^{\mathrm{r}}}$.
    
    \item At the $i$th iteration, we define $\mathfrak{G}_\mathscr{R}^{(i+1)}$ from $\mathfrak{G}_{\mathscr{R}}^{(i)}$ as follows: if $\#\mathfrak{G}_\mathscr{R}^{(i)} = n$, then set $\mathfrak{G}_{\mathscr{R}}^{(i+1)} = \mathfrak{G}_\mathscr{R}^{(i)}$; Otherwise $\mathfrak{G}_{\mathscr{R}}^{(i)}$ is not a dualizable $\rpoint$-dissection of $(\pmb{\Sigma},\mathcal{M})$, and so:
    \begin{enumerate}[label=$(2 \alph*)$,itemsep=1mm]
    \item select a connected component $C_i \in \bbGamma(\mathfrak{G}_\mathscr{R}^{(i)})$, a pair $(v_i,w_i) \in (\mathcal{M}_{\gpoint} \cap \partial C_i)^2$ such that $v_i\neq w_i$, and a projective accordion $\rho_i \in (v_i | w_i)_{\Prj}$.

    \item If $\rho_i$ has extensions with two accordions in $\partial C$ (Configurations \ref{2} of \cref{cor:wheretoseparatecellsallcases}): denoted by them $\delta$ and $\eta$, the pair $(\delta,\eta)$ admits a $\CoZ$-completion $\mu_i$. We set $\mathfrak{G}_\mathscr{R}^{(i+1)}=\mathfrak{G}_\mathscr{R}^{(i)}\cup\lbrace \mu_i\rbrace$;

    \item otherwise, $\rho_i$ has extensions with only one accordion in $\partial C$ (Configurations \ref{1} of \cref{cor:wheretoseparatecellsallcases}): by denoting it $\delta$ there exists an endpoint $e_i$ of $\rho_i$ in $\partial C$ such that $\wt_{e_i}(\rho_i,\mu)=0$ for any $\mu \in (\mathfrak{G}_\mathscr{R}^{(i)})_{[e_i]}$. Fix a coloration of $\delta$ such that $e_i$ is below $\delta$. Consider $\nu_i$ such that $s(\nu_i)=\s_{(\delta)}(e_i)$ and $t(\nu_i)=e_i$. Set $\mathfrak{G}_\mathscr{R}^{(i+1)}=\mathfrak{G}_\mathscr{R}^{(i)}\cup\lbrace \nu_i\rbrace$.
    \end{enumerate}
    
    \item Go back to Step $(2)$ while $\mathfrak{G}_{\mathscr{R}}^{(i+1)} \neq \mathfrak{G}_{\mathscr{R}}^{(i)}$;
    
    \item return $\mathfrak{G}_{\mathscr{R}}^{(i+1)}$ otherwise.
\end{enumerate}
\end{algo}

\begin{theorem} \label{thm:Tiltgenerator}
Let $(Q,R)$ be a gentle tree, and $\mathscr{R}$ be a resolving subcategory. Then \cref{algo:tiltgeom} ends and returns $\mathfrak{G} \subset \Accord'$ such that both:
\begin{enumerate}[label=$(\roman*)$, itemsep=1mm]
    \item $\mathfrak{T} = \{\MM(\delta) \mid \delta \in \mathfrak{G}\}$ is a tilting collection of $(Q,R)$; and,
    \item $\Res(\mathfrak{T}) = \mathscr{R}$.
\end{enumerate}
Therefore, $\mathfrak{T} = \mathfrak{T}_{\mathscr{R}}$ is the collection defined via the Auslander--Reiten correspondence recalled in \cref{thm:Auslander}.
\end{theorem}

\begin{proof}
This is a direct consequence of \cref{cor:augdissecrigid}. 
\end{proof}

\begin{cor} \label{cor:uniqucanonrigidR}
Let $(Q,R)$ be a gentle tree, and $\mathscr{R}$ be a resolving subcategory. Then $\mathfrak{E}_{\mathscr{R}}^{\mathrm{r}}$ is the unique subcollection of $\mathscr{R}$ such that the following assertions hold:
\begin{enumerate}[label=$(\roman*)$, itemsep=1mm]
    \item  $\mathfrak{E}_\mathscr{R}^{\mathrm{r}}$ is an antichain of $(\pmb{\ind \setminus \proj}(Q,R), \Resleq)$;
    \item $\mathfrak{E}_\mathscr{R}^{\mathrm{r}}$ is rigid; and,
    \item $\Res(\mathfrak{E}_\mathscr{R}^{\mathrm{r}})=\mathscr{R}$.
\end{enumerate}
\end{cor}

\subsection{An example}
\label{ss:examplePart2}

Let us consider $(Q,R)$ the gentle tree shown in \cref{fig:ExsurfTilt}. We also give its Res-poset $(\pmb{\ind \setminus \proj}, \Resleq)$, and its $\rpoint$-dissected marked surface $(\pmb{\Sigma}, \mathcal{M}, \Prj(\Delta^{\gpoint}))$.

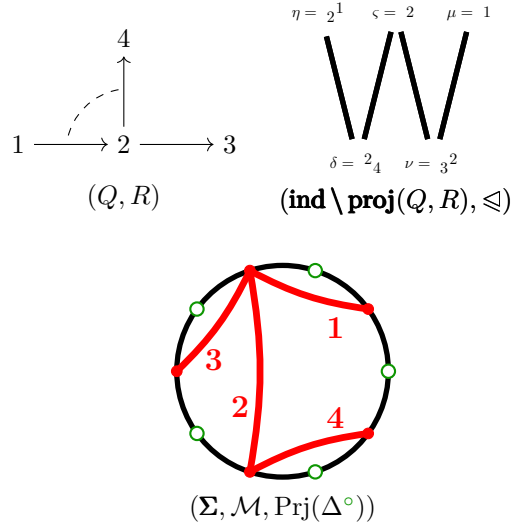
\begin{figure}[!ht]
\centering 
    \begin{tikzpicture}
    \begin{scope}[->,scale=1.4]
		\node (a) at (0,0) {$1$};
		\node (b) at (1,0) {$2$};
		\node (c) at (2,0) {$3$};
		\node (d) at (1,1) {$4$};
		
		\draw (a) -- (b);
		\draw (b) -- (c);
		\draw (b) -- (d);

		\draw[dashed,-] ([xshift=-.05cm,yshift=.35cm]b.north) arc[start angle = 100, end angle = 170, x radius=.6cm, y radius =.6cm];
		
		\node at (1,-.5) {$(Q,R)$};
    \end{scope}
    
    \begin{scope}[xshift=5cm,yshift=-1.25cm]
	    \node (b) at (-0.5,1) {\scalebox{0.6}{$\delta =
	    \begin{tikzpicture}[baseline={(0,-.2)},scale=0.2]
	    \node at (0,0){$2$};
	    \node at (1,-1){$4$};
	    \end{tikzpicture}$}};
	    \node (c) at (0.5,1)  {\scalebox{0.6}{$ \nu = 
	    \begin{tikzpicture}[baseline={(0,-.2)},scale=0.2]
	    \node at (0,0){$2$};
	    \node at (-1,-1){$3$};
	    \end{tikzpicture}$}};
	    \node (e) at (0,3) {\scalebox{0.6}{$ \varsigma =
	    \begin{tikzpicture}[baseline={(0,-.1)},scale=0.2]
	    \node at (0,0){$2$};
	    \end{tikzpicture}$}};
	    \node (f) at (-1,3) {\scalebox{0.6}{$ \eta =
	    \begin{tikzpicture}[baseline={(0,-.2)},scale=0.2]
	    \node at (0,0){$1$};
	    \node at (-1,-1){$2$};
	    \end{tikzpicture}$}};
	    \node (h) at (1,3) {\scalebox{0.6}{$\mu = \begin{tikzpicture}[baseline={(0,-.1)},scale=0.2]
	    \node at (0,0){$1$};
	    \end{tikzpicture}$}};

		\draw[line width=0.7mm,black] (b) -- (e);
		\draw[line width=0.7mm,black] (b) -- (f);
		\draw[line width=0.7mm,black] (c) -- (e);
		\draw[line width=0.7mm,black] (c) -- (h);
		
		\node at (0,.5) {$(\pmb{\ind \setminus \proj}(Q,R), \Resleq)$};

		\draw[line width=0.7mm,black] (b) -- (e);
		\draw[line width=0.7mm,black] (b) -- (f);
		\draw[line width=0.7mm,black] (c) -- (e);
		\draw[line width=0.7mm,black] (c) -- (h);
		
		\node at (0,.5) {$(\pmb{\ind \setminus \proj}(Q,R), \Resleq)$};
	\end{scope}
	\begin{scope}[xshift=3.5cm,yshift=-3cm,mydot/.style={
					circle,
					thick,
					fill=white,
					draw,
					outer sep=0.5pt,
					inner sep=1pt
				}, scale = 0.7]
		\tikzset{
		osq/.style={
        rectangle,
        thick,
        fill=white,
        append after command={
            node [
                fit=(\tikzlastnode),
                orange,
                line width=0.3mm,
                inner sep=-\pgflinewidth,
                cross out,
                draw
            ] {}}}}
		\draw[line width=0.7mm,black] (0,0) ellipse (2cm and 2cm);
		\foreach \X in {0,1,...,9}
		{
		\tkzDefPoint(2*cos(pi/5*\X),2*sin(pi/5*\X)){\X};
		};

		\draw[line width=0.9mm ,bend left =10,red](1) edge (3);
		\draw[line width=0.9mm ,bend left =10,red](3) edge (7);
		\draw[line width=0.9mm ,bend left =10,red](3) edge (5);
		\draw[line width=0.9mm ,bend left =10,red](7) edge (9);
		
	
		

		\foreach \X in {1,3,5,7,9}
		{
		\tkzDrawPoints[fill =red,size=4,color=red](\X);
		};
		
		\foreach \X in {0,2,4,6,8}
		{
		\tkzDrawPoints[size=5,dark-green,mydot](\X);
		};

		\tkzDefPoint(1,1.2){gamma};
		\tkzLabelPoint[red](gamma){\Large $\mathbf{1}$}
		\tkzDefPoint(-0.8,-0.3){gamma};
		\tkzLabelPoint[red](gamma){\Large $\mathbf{2}$}
		\tkzDefPoint(-1.3,0.6){gamma};
		\tkzLabelPoint[red](gamma){\Large $\mathbf{3}$}
		\tkzDefPoint(1,-0.5){gamma};
		\tkzLabelPoint[red](gamma){\Large $\mathbf{4}$}
		\node at (0,-2.6) {$(\pmb{\Sigma}, \mathcal{M}, \Prj(\Delta^{\gpoint}))$};
	\end{scope}
    \end{tikzpicture}
\caption{\label{fig:ExsurfTilt} A gentle tree $(Q,R)$ with the Res-relation order on its non-projective indecomposable representations, and the dualizable $\rpoint$-dissected marked surface associated to $(Q,R)$.}
\end{figure}
We identify the notation for strings and the one for accordions. In what follows, denote by $\rho_i$ the projective accordion corresponding to the projective representation $P_i$ for all $i \in \{1,2,3,4\}$. 

\begin{ex}\label{ex:Tiltcalc1}
In \cref{fig:Exsurf4}, we compute the two tilting representations from the two resolving subcategories already calculated in \cite[Section 6.3]{DS252}. Both are obtained by only taking the projective accordions that are $\Ext$-ortogonal to the antichain that generates the resolving sets as ideals in $(\Accord', \Accordleq)$. \qedhere
\begin{figure}[!ht]
\centering 
    \begin{tikzpicture}
	\begin{scope}[xshift=5cm,yshift=-2cm,mydot/.style={
					circle,
					thick,
					fill=white,
					draw,
					outer sep=0.5pt,
					inner sep=1pt
				}, scale = 0.7]
		\tikzset{
		osq/.style={
        rectangle,
        thick,
        fill=white,
        append after command={
            node [
                fit=(\tikzlastnode),
                orange,
                line width=0.3mm,
                inner sep=-\pgflinewidth,
                cross out,
                draw
            ] {}}}}
		\draw[line width=0.7mm,black] (0,0) ellipse (2cm and 2cm);
		\foreach \X in {0,1,...,9}
		{
		\tkzDefPoint(2*cos(pi/5*\X),2*sin(pi/5*\X)){\X};
		};
		\draw[line width=0.9mm ,bend left =10,red](1) edge (3);
		\draw[line width=0.9mm ,bend left =10,red](3) edge (5);
		
		\draw[line width=0.7mm ,bend right=10,blue, loosely dashed](1) edge (9);
		\draw[line width=0.7mm ,bend right=-10,blue, loosely dashed](5) edge (7);
		

		\foreach \X in {1,3,5,7,9}
		{
		\tkzDrawPoints[fill =red,size=4,color=red](\X);
		};
		
		
		\foreach \X in {0,2,4,6,8}
		{
		\tkzDrawPoints[size=5,dark-green,mydot](\X);
		};

		\tkzDefPoint(-.8,-0.5){gamma};
		\tkzLabelPoint[blue](gamma){\Large $\delta$}
		\tkzDefPoint(-.9,1){j};
		\tkzLabelPoint[red](j){\Large $\rho_3$}
		\tkzDefPoint(1,-0.2){s};
		\tkzLabelPoint[blue](s){\Large $\mu$}
		\tkzDefPoint(0.3,1.3){t};
		\tkzLabelPoint[red](t){\Large $\rho_1$}
	\end{scope}
	\begin{scope}[xshift=0cm,yshift=-2cm,mydot/.style={
					circle,
					thick,
					fill=white,
					draw,
					outer sep=0.5pt,
					inner sep=1pt
				}, scale = 0.7]
		\tikzset{
		osq/.style={
        rectangle,
        thick,
        fill=white,
        append after command={
            node [
                fit=(\tikzlastnode),
                orange,
                line width=0.3mm,
                inner sep=-\pgflinewidth,
                cross out,
                draw
            ] {}}}}
		\draw[line width=0.7mm,black] (0,0) ellipse (2cm and 2cm);
		\foreach \X in {0,1,...,9}
		{
		\tkzDefPoint(2*cos(pi/5*\X),2*sin(pi/5*\X)){\X};
		};
		\draw[line width=0.9mm ,bend left =10,red](1) edge (3);
		\draw[line width=0.9mm ,bend left =10,red](7) edge (9);
		
		\draw[line width=0.7mm ,bend right=10,blue, loosely dashed](1) edge (5);
		\draw[line width=0.7mm ,bend right=10,blue, loosely dashed](5) edge (9);
		

		\foreach \X in {1,3,5,7,9}
		{
		\tkzDrawPoints[fill =red,size=4,color=red](\X);
		};
		
		
		\foreach \X in {0,2,4,6,8}
		{
		\tkzDrawPoints[size=5,dark-green,mydot](\X);
		};

		\tkzDefPoint(-1,0.5){gamma};
		\tkzLabelPoint[blue](gamma){\Large $\eta$}
		\tkzDefPoint(-.4,-1){j};
		\tkzLabelPoint[red](j){\Large $\rho_4$}
		\tkzDefPoint(0,-0.2){s};
		\tkzLabelPoint[blue](s){\Large $\varsigma$}
		\tkzDefPoint(-.4,1.6){t};
		\tkzLabelPoint[red](t){\Large $\rho_1$}
	\end{scope}
    \end{tikzpicture}
\caption{\label{fig:Exsurf4} The geometric tilting collections $\mathfrak{T}_1 = \{\eta, \varsigma,\rho_1, \rho_4\}$, and $\mathfrak{T}_2 = \{\delta, \mu,\rho_1, \rho_3\}$ associated to the geometric resolving subcategory $\opResAc(\eta,\varsigma) = \opResAc(\eta) \cup \opResAc(\varsigma)$ and $\opResAc(\delta,\mu) = \opResAc(\delta) \cup \opResAc(\mu)$ calculated in \cite[Section 6.3]{DS252}.}
\end{figure}
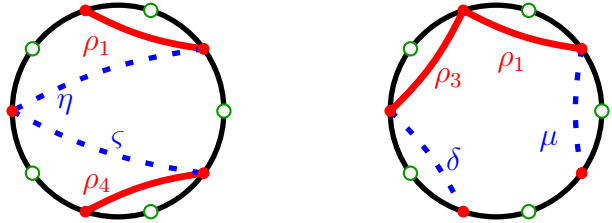
\end{ex}

\begin{ex} \label{ex:Tiltcalc2} See \cref{fig:Exsurf6} for an illustration. Consider the resolving set obtained by taking all the accordions of $(\pmb{\Sigma}, \mathcal{M}, \Delta^{\gpoint})$. Then $\mathfrak{E}_{\Accord'} = \left\{\eta, \varsigma, \mu \right\}$. Then $\mathfrak{E}_{\Accord'}$ is not rigid. Using \cref{algo:tiltgeom}, we obtain $\{\eta, \varsigma\}$. 

Here, taking the projective accordions that are $\Ext$-orthogonal to $\{\eta,\varsigma\}$ is not enough to obtain the desired tilting representation; we must also add $\rho_1$ to our rigid collection. So $\widehat{\mathfrak{D}} = \{\eta,\varsigma,\rho_1\}$ is an admissible $\rpoint$-dissection of $(\pmb{\Sigma}, \mathcal{M})$ which is not dualizable.

 Let us denote by $v$ and $w$ the two green-marked points that are not separated in $\widehat{\mathfrak{D}}$. Then $(v|w)_{\Prj} = \{\rho_2,\rho_4\}$. We note that $\rho_2$ crosses $\eta$ and so, using the laterality induced by $\varsigma$, we set $\xi$ to be the accordion such that $s(\xi) = s(\eta)$ and $t(\xi) = t(\rho_2)$. Here $\xi = \delta$. Note that if we use $\rho_4$, then the accordion $\xi$ is such that $M(\xi) \in \Ext^2(\MM(\eta), \MM(\rho))$. Moreover $\Ext^i(\MM(\eta), \MM(\rho)) = 0$ for $i > 3$. So either way, we get that $\xi = \delta$.

In the end, we get that $\mathfrak{T} = \{\eta,\varsigma,\delta,\rho_1\}$ is the tilting collection associated to the resolving set $\Accord'$. \qedhere

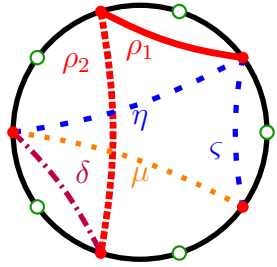
\begin{figure}[!ht]
\centering 
    \begin{tikzpicture}
	\begin{scope}[scale=1.2,mydot/.style={
					circle,
					thick,
					fill=white,
					draw,
					outer sep=0.5pt,
					inner sep=1pt
				}, scale = 0.7]
		\tikzset{
		osq/.style={
        rectangle,
        thick,
        fill=white,
        append after command={
            node [
                fit=(\tikzlastnode),
                orange,
                line width=0.3mm,
                inner sep=-\pgflinewidth,
                cross out,
                draw
            ] {}}}}
		\draw[line width=0.7mm,black] (0,0) ellipse (2cm and 2cm);
		\foreach \X in {0,1,...,9}
		{
		\tkzDefPoint(2*cos(pi/5*\X),2*sin(pi/5*\X)){\X};
		};
		
		\draw[line width=0.9mm ,bend left =10,red](1) edge (3);
		\draw[line width=0.9mm ,bend left =10,red,densely dashdotted](3) edge (7);
		
		\draw[line width=0.7mm ,bend right=10,blue, loosely dashed](1) edge (9);
		\draw[line width=0.7mm ,bend right=10,blue,loosely dashed](5) edge (1);

		\draw[line width=0.7mm ,bend left=10,orange, loosely dotted](5) edge (9);
		
		\draw[line width=0.7mm ,bend right=-10,purple, dash pattern={on 5pt off 2pt on 1pt off 2pt}](5) edge (7);
		

		\foreach \X in {1,3,5,7,9}
		{
		\tkzDrawPoints[fill =red,size=4,color=red](\X);
		};
		
		
		\foreach \X in {0,2,4,6,8}
		{
		\tkzDrawPoints[size=5,dark-green,mydot](\X);
		};

		\tkzDefPoint(-.9,-0.3){gamma};
		\tkzLabelPoint[purple](gamma){\Large $\delta$}
		\tkzDefPoint(1.2,0){i};
		\tkzLabelPoint[blue](i){\Large $\varsigma$}
		\tkzDefPoint(0,-0.4){j};
		\tkzLabelPoint[orange](j){\Large $\mu$}
		\tkzDefPoint(0,0.5){k};
		\tkzLabelPoint[blue](k){\Large $\eta$}
		\tkzDefPoint(0,1.6){t};
		\tkzLabelPoint[red](t){\Large $\rho_1$}
            \tkzDefPoint(-1,1.4){v};
		\tkzLabelPoint[red](v){\Large $\rho_2$}
	\end{scope}
    \end{tikzpicture}
\caption{\label{fig:Exsurf6} The construction of the geometric tilting collection $\mathfrak{T} = \{\eta, \varsigma, \delta, \rho_1\}$ associated to the geometric resolving subcategory given by taking the set of all accordions $\Accord$.}
\end{figure}
\end{ex}

	\section*{Acknowledgements}
	
    B.D. thanks the \emph{Institut des Sciences Mathematiques (UQAM)} and the \emph{Engineering and Physical Sciences Research Council (EP/W007509/1)} for their partial funding support. The authors acknowledge the \emph{CHARMS program grant (ANR-19-CE40-0017-02)} for their funding support.
    
	The authors thank Claire Amiot, Frédéric Chapoton, Yann Palu and Baptiste Rognerud,  for their discussions and advice throughout this work.

	\bibliography{Article}
	\bibliographystyle{alpha}
	
\end{document}